\documentclass[leqno]{amsart}
\usepackage{amsfonts,amssymb,amsmath,amsgen,amsthm, mathrsfs}
\usepackage{hyperref,color}
\usepackage{pdfsync}
\usepackage{bm}
\usepackage{bbm}

\theoremstyle{plain}
\newtheorem{theorem}{Theorem}[section]

\newtheorem{lemma}[theorem]{Lemma}

\newtheorem{proposition}[theorem]{Proposition}

\theoremstyle{remark}
\newtheorem{remark}[theorem]{Remark}

\def\C{{\mathbb C}}
\def\R{{\mathbb R}}
\def\N{{\mathbb N}}
\def\Z{{\mathbb Z}}

\def\G{\mathcal G}

\newcommand{\ii}{\mathrm{i}}
\newcommand{\ud}{\mathrm{d}}
\newcommand{\overbar}[1]{\mkern 1.5mu\overline{\mkern-1.5mu#1\mkern-1.5mu}\mkern 1.5mu}

\def\({\left(}
\def\){\right)}
\def\<{\left\langle}
\def\>{\right\rangle}

\def\1{{\mathbf 1}}

\def\eps{\varepsilon}

\def\om{\omega}

\def\D{{\mathfrak{D}}}
\def\deb{{\rightharpoonup}}

\numberwithin{equation}{section}

\date\today

\title[Point interactions and singular solutions]{Point interactions and singular solutions to semilinear elliptic equations}

\author{Filippo Boni}
\address{Politecnico di Torino, Dipartimento di Scienze Matematiche ``G.L.~Lagrange'', Corso Duca degli Abruzzi 24, 10129 Torino, Italy}
\email{filippo.boni@polito.it}

\author[Diego Noja]{Diego Noja}
\address{Dipartimento di Matematica e Applicazioni, Universit\`a di Milano Bicocca, Via R. Cozzi 55, 20126 Milano, Italy}
\email{diego.noja@unimib.it}

\author[Raffaele Scandone]{Raffaele Scandone}
\address{Universit\`a degli Studi di Napoli ``Federico II'', Dipartimento di Matematica e Applicazioni ``R.~Caccioppoli'', Complesso Monte S.~Angelo - Via Cintia, 80126 Napoli, Italy}
\email{raffaele.scandone@unina.it}

\begin{document}
	\begin{abstract} 
We investigate the connection between semilinear elliptic PDEs with isolated singularities and stationary nonlinear Schrödinger equations with point interactions. In dimensions $d=2,3$, we establish a rigorous correspondence between their solutions, revealing two regimes depending on whether a boundary condition at the singularity can be imposed. This connection enables us to exploit operator-theoretic and variational methods that have not previously been applied to the study of isolated singularities. In the source regime, we prove the existence of infinitely many radial singular solutions, by applying the symmetric mountain pass theorem of Ambrosetti and Rabinowitz to the action functional associated with the point interaction. When $d=2$, a suitable uniqueness result allows us to characterize singular ground states (positive solutions) as action minimizers and to prove the existence of infinitely many nodal singular solutions.
	\end{abstract}
	\date{\today}
	
	\subjclass[2020]{35A15, 35J61, 35Q40, 81Q10, 35B09}

	\keywords{Semilinear elliptic equations, singular solutions, point interactions, variational methods, nodal solutions}
	
\maketitle

\section{Introduction}
We consider the semilinear elliptic equation
\begin{equation}\label{eq:singular_elliptic}
(-\Delta+\lambda)u=\sigma|u|^{p-1}u,\quad x\in\R^d\setminus\{0\},
\end{equation}
in dimensions $d=2,3$, where $\sigma=\pm 1$, $\lambda>0$ and $p>1$. We emphasize that the equation is posed on the punctured domain
$\R^d\setminus\{0\}$ rather than on the whole space $\R^d$, so that solutions are allowed to exhibit a non-regular behavior at the origin. We focus on \emph{real-valued} solutions, even though most of our analysis can be readily extended to the complex-valued case, see Section \ref{re:complesso}.

Equation \eqref{eq:singular_elliptic}, as well as several of its variants (starting with the case $\lambda=0$) and generalizations, has been widely investigated; see, for instance, \cite{Lions80, Veron81, NS1_86, NS2_86, JPY94, Veron08, Naito-Sato, CirsteaDu07, GKS}, the monograph \cite{Veron-book}, and the references therein.  In this context, the positive and negative signs in front of the nonlinearity are commonly referred to as \emph{source term} and \emph{absorption term}, respectively. In the related setting of time-dependent nonlinear Schr\"odinger equations, the same distinction is typically expressed using the terminology \emph{focusing} and \emph{defocusing} nonlinearities. 

The above-cited works are mainly concerned with the classification of the possible singular behavior at the origin, in analogy with the well-known linear case $\sigma=0$. More precisely, under suitable assumptions, solutions to \eqref{eq:singular_elliptic} are known to be either regular, namely, they extend to $\mathcal{C}^2(\mathbb{R}^d)$ functions, or singular, namely, they belong to $\mathcal{C}^2(\mathbb{R}^d\setminus{0})$ and satisfy $|u(x)|\to\infty$ as $|x|\to0$. In the former case, the singularity is said to be \emph{removable}, whereas in the latter, which is our main focus, it is said to be \emph{isolated}.

The classical literature on the subject has largely focused on positive solutions with an isolated singularity at the origin and vanishing at infinity, the so-called singular ground states. These solutions can be classified according to the sign of $\sigma$, the exponent $p$ of the nonlinearity, and the space dimension $d$ (see the discussion in Section~\ref{isolatedsol}).

Our aim is to provide a novel \emph{operator-theoretic} and \emph{variational} characterization of equation~\eqref{eq:singular_elliptic}, that applies to both positive and sign-changing solutions, and to show that this framework leads to new insights and results concerning isolated singularities in nonlinear elliptic problems.

The key fact underlying our analysis is the connection between elliptic equations with isolated singularities and Schr\"odinger operators with a \emph{point interaction}. The latter can be seen as singular perturbations of the Laplacian, sometimes (informally and somewhat misleadingly) described as $\delta$-type potentials. More precisely, they are rigorously realized as the family of self-adjoint extensions of the Laplacian initially defined on $\mathcal{C}_0^\infty(\R^d\setminus\{0\})$, and are commonly denoted by $-\Delta_\alpha$, where $\alpha$ is the parameter labelling the specific extension (see Section \ref{subsec:Point} for a detailed description). We point out that in dimension $d>3$, the only possible extension is the free Laplacian,  which is the reason for the limitation on the dimension we imposed in \eqref{eq:singular_elliptic} (see Section \ref{re:dprime} for further comments). Schr\"odinger operators with point interactions are widely used in quantum mechanics, where they provide solvable effective models for the interaction of a non-relativistic particle with a very short-range potential (the so called ``zero range'' interactions). We refer to \cite{albeverio-solvable,alb-fig} and the references therein for a comprehensive overview of point interactions, covering both their rigorous mathematical treatment and their role in physics.

The strong connection between the structure of the domain of point-interaction operators and the local behavior of solutions to \eqref{eq:singular_elliptic} near isolated singularities was first emphasized in \cite{CFN-21}, in the context of well-posedness for a time-dependent Schrödinger equation. More specifically, functions belonging to the domain of $-\Delta_{\alpha}$ admit a decomposition into a regular part and a singular component, the latter being given by a multiple of the Green function of the Laplace operator. A similar structure is exhibited by solutions to \eqref{eq:singular_elliptic} in suitable regimes, where the singular behavior at the origin is described by the same Green function profile.

This analogy naturally suggests a correspondence between singular solutions to \eqref{eq:singular_elliptic} and solutions to the semilinear elliptic problem
\begin{equation}\label{eq:point_elliptic1}
	(-\Delta_{\alpha}+\lambda)u=\sigma|u|^{p-1}u,
\end{equation}
where the Laplacian is replaced by a point-interaction operator $-\Delta_{\alpha}$. 

Our first goal is to establish this correspondence in a rigorous and complete way, showing that the point-interaction formulation provides indeed an operator-theoretic realization of the singular behavior arising in \eqref{eq:singular_elliptic}.

The correspondence is formulated in Theorems \ref{th:main_equivalence} and \ref{th:equivalence-weak}, which identify two qualitatively different regimes. In the setting of Theorem \ref{th:main_equivalence}, referred to as the \emph{Trace Regime}, the regular part is continuous. In this case, the relation between the coefficient of the singular component and the value of the regular part at $x=0$ exactly reproduces the boundary condition defining the domain of $\Delta_\alpha$, for a unique choice of the parameter $\alpha\in\mathbb R$. Conversely, in the regime described by Theorem \ref{th:equivalence-weak}, referred to as the \emph{Free Regime}, the regular part is itself singular at the origin. As a consequence, no trace condition can be imposed and no distinguished value of $\alpha$ is selected. 

In the Trace Regime, there is also a natural variational setting associated with \eqref{eq:point_elliptic1}, which is rigorously described in Theorem \ref{th:main_equivalence}. 

Our second goal is to exploit this variational framework to investigate the existence and qualitative properties of solutions to \eqref{eq:singular_elliptic}. Restricting to the source case, we first prove Theorem \ref{th:main_critical_points} on the existence of infinitely many radial singular solutions, by applying the mountain-pass method of Ambrosetti-Rabinowitz to the action functional associated with the point interaction. This result is of independent interest, as it allows us to recover classical results through variational arguments. Moreover, in the case $d=2$, we obtain two further relevant consequences. The first one, Theorem \ref{th:structure}, is a structure result, identifying singular ground states (i.e., positive solutions to \eqref{eq:singular_elliptic}) with minimizers of the action. The second one, Theorem \ref{th:nodal}, gives the existence of infinitely many singular \emph{nodal} (sign-changing) solutions. To the best of our knowledge, these results are new and suggest that the connection between problems \eqref{eq:singular_elliptic} and \eqref{eq:point_elliptic1} is both effective and promising.

\subsection*{Structure of the paper}
In Section~\ref{sec:mainresults} we state our main theorems, after collecting the necessary preliminaries on point interactions. We also provide an overview of classical results on elliptic equations with isolated singularities. Section~\ref{sec:Proofs} is devoted to the proofs of the main theorems, preceded by a series of preliminary technical lemmas. In Section~\ref{sec:finalremarks} we collect some further remarks and discuss possible extensions of our framework. Finally, in Appendix~\ref{app} we present the proof of a technical lemma, due to Brezis and Lions, in the form required for our purposes.


\subsection*{Notation}
\begin{itemize}
\item[-] We denote by $C$ a positive constant, which may vary from line to line. 
\item[-] Given two positive quantities $X$, $Y$, we write $X\lesssim Y$ if $X\leqslant CY$. 
\item[-]Given real valued functions $f$ and $g$, we write $f\sim g$ as $x\to x_0$ when $\lim_{x\to x_0}g^{-1}f\neq 0$ (finite), and $f=o(g)$ as $x\to x_0$ when $\lim_{x\to x_0}g^{-1}f=0$. 
\item[-] Given $x\in\R$, the symbols $x^-$ and $x^+$ mean, respectively, $x-\eps$ and $x+\eps$ for every $\eps>0$ sufficiently small; similarly $\infty-$ means ``every $p$ sufficiently large''. 
\item[-] Given a real valued function $f$, we denote respectively by $f^{+}:=\max\{f,0\}$ and $f^{-}:=\max\{-f,0\}$ its positive and negative parts, so that $f=f^+-f^-$. 
\item[-] The symbols $L^p(\R^d)$ and $H^{s,p}(\R^d)$ denote the usual Lebesgue and Sobolev (Bessel-potential) spaces of real-valued functions on $\R^d$, with the customary convention that $H^{0,p}(\R^d):=L^p(\R^d)$, $H^s(\R^d):=H^{s,2}(\R^d)$. 
\item[-] For function spaces, the subscripts $0$, $b$, and $\mathrm{rad}$ denote, respectively, compact support, boundedness, and radial symmetry.  
\item[-] Given a Banach space $(X,\|\cdot\|)$ and $R>0$, we set $B_R:=\{x\in X\,|\,\|x\|\leqslant R\}$, $B_R^*=B_R\setminus\{0\}$, $\partial B_R=\{x\in X\,|\,\|x\|=R\}$.
\item[-] We write $X'$ for the topological dual of $X$, and denote by $\langle \cdot,\cdot\rangle_{X\times X'}$ the duality product, omitting the subscript when we consider the standard $L^2\times L^2$ product.
\end{itemize}
 
\section{Main results}\label{sec:mainresults}
 
\subsection{Point interactions}\label{subsec:Point} In order to provide a precise statement of our results, we first recall the rigorous construction and main features of Schr\"odinger operators with point interactions, based on \cite[Chapters I.1, I.5]{albeverio-solvable}. Although the results in \cite{albeverio-solvable} are formulated for complex-valued functions, they remain valid in the real-valued setting, see Remark \ref{re:est-re} below.

The densely defined, symmetric operator $-\Delta_{|\mathcal{C}_0^{\infty}(\R^d\setminus\{0\};\,\R)}$, when $d=2,3$, admits a one-parameter family $-\Delta_{\alpha}$, $\alpha\in\overline{\R}=\R\cup\{\infty\}$, of semi-bounded below, self-adjoint extensions on $L^2(\R^d)$. The extension corresponding to $\alpha=\infty$ is the free Laplace operator (with $H^2(\R^d)$ as domain of self-adjointness), whereas all the others, with $\alpha\in\R$, realize non-trivial point interactions at the origin.

Given $\omega>0$, and denoting by $\mathcal{G}_{\omega}:=(-\Delta+\omega)^{-1}\delta\in L^2(\R^d)$ the Green kernel (a.k.a.~Green function) of $-\Delta+\omega$, the domain and action of $-\Delta_{\alpha}$, with $\alpha\in\R$, are given respectively by
\begin{gather}\label{domain}\mathcal{D}(-\Delta_{\alpha})=\big\{u\in L^2(\R^d)\,:\,u=f+q\mathcal{G}_{\omega},\, f\in H^2(\R^d),\,\beta_{\alpha}(\omega)q=f(0)\big\},\\
\label{action}(-\Delta_{\alpha}+\omega)u=(-\Delta+\omega)f,
\end{gather}
where 
\begin{equation}\label{eq:beta_alfa}
\beta_{\alpha}(\omega)=\begin{cases}
\alpha+\frac{\gamma}{2\pi}+\frac{1}{2\pi}\ln\frac{\sqrt{\omega}}{2}&d=2\\
\alpha+\frac{\sqrt{\omega}}{4\pi}&d=3,
\end{cases}
\end{equation}
and $\gamma\approx0.577$ is the Euler-Mascheroni constant.
\begin{remark}\label{re:est-re}
In \cite{albeverio-solvable} it is proved that \eqref{domain}-\eqref{action} above, in the complex-valued setting, provide the non-trivial self-adjoint extensions on $L^2(\R^d;\C)$ of $-\Delta_{|\mathcal{C}_0^{\infty}(\R^d\setminus\{0\};\,\C)}$. However these operators are \emph{real}, in the sense that both their domain and action are invariant under complex conjugation. Therefore, their restriction to real-valued functions yields precisely the non-trivial self-adjoint extensions on $L^2(\R^d;\R)$ of $-\Delta_{|\mathcal{C}_0^{\infty}(\R^d\setminus\{0\};\,\R)}$.
\end{remark}
We briefly comment on the above characterization. Any element of the domain is the sum of a {\em regular part}, $f\in H^2(\R^d)$, and a {\em singular part}, $q\mathcal{G}_{\omega}$, proportional to the Green kernel of the Laplacian. Moreover, a linear relation exists between the coefficient $q$ of the singular part, usually called \emph{charge}, and the evaluation at the singularity point (i.e.~at $x=0$) of the regular part. In \eqref{domain}, the \emph{regular} component $f$ may depend on $\omega>0$, whereas the charge $q$ is independent of $\omega$, since $\G_{\omega_1}-\G_{\omega_2}\in H^2(\R^d)$ for any $\omega_1\neq\omega_2$. We point out that the validity of the boundary condition $\beta_{\alpha}(\omega)q=f(0)$ is also independent on the choice of $\omega>0$, for $(\G_{\omega_1}-\G_{\omega_2})(0)=\beta_{\alpha}(\omega_2)-\beta_{\alpha}(\omega_1)$ for any $\omega_1\neq\omega_2$. The above information makes it clear that the domain of $-\Delta_{\alpha}$ is independent of the auxiliary parameter $\omega>0$. Regarding the action of the operator, $-\Delta_{\alpha}+\omega$ acts on a domain element $u$ as the free operator $-\Delta+\omega$ acts on its regular part $f$. This entails that any $u\in \mathcal{D}(-\Delta_{\alpha})$ satisfies
\begin{equation}\label{eq:dirac_elliptic}
-\Delta_\alpha u=-\Delta u-q\delta_0 \quad\text{in}\quad \mathcal{D}'(\R^d),  
\end{equation} 
that also explains the name {\em $\delta$ interaction} or {$\delta$ potential}  sometimes used in the literature (observe however that $q$ depends on $u$, so that $q\delta_0$ is not properly a source).

We recall for future use that
$\mathcal{G}_{\omega}$ is positive, radially decreasing, exponentially decaying at infinity, and smooth outside the origin, with
\begin{equation}\label{eq:asy}
\mathcal{G}_{\omega}(x)\sim\begin{cases}
-\dfrac{1}{2\pi}\ln |x|&d=2\\[0.8em]
\dfrac{1}{4\pi |x|}&d=3
\end{cases}\qquad\mbox{as}\quad x\to 0.
\end{equation}
In particular,
\begin{equation}\label{lpgreen}
\mathcal{G}_{\omega}\in L^p(\R^d)\Leftrightarrow\begin{cases}
p\in[1,\infty)\, \ \ d=2\\[0.4em]
p\in[1,3)\,\ \ \ \ d=3.
\end{cases}
\end{equation}

In view of \eqref{domain}, domain elements of $-\Delta_{\alpha}$ may be regular, i.e.~functions in $H^2(\R^d)$, if and only if $q=0$.

The spectrum of $-\Delta_{\alpha}$ is characterized as follows:
\begin{gather*}
\sigma_{\operatorname{ess}}(-\Delta_{\alpha})=\sigma_{\operatorname{ac}}(-\Delta_{\alpha})=[0,+\infty),\qquad \sigma_{\operatorname{sc}}(-\Delta_{\alpha})=\emptyset \\
\sigma_{\operatorname{p}}(-\Delta_{\alpha})=\begin{cases}
	\{-\lambda_{\alpha}\}&(d=2,\,\alpha\in\R)\mbox{ or }(d=3,\,\alpha<0)\\
	\emptyset&(d=3,\,\alpha\geqslant 0),
\end{cases}
\end{gather*}
where $\lambda_{\alpha}$ is the unique positive solution to $\beta_{\alpha}(\lambda_\alpha)=0$, which exists if and only if $(d=2$, $\alpha\in\R)$ or $(d=3$, $\alpha <0)$. The negative eigenvalue $-\lambda_{\alpha}$, when it exists, is simple, and has $\mathcal{G}_{\lambda_{\alpha}}$ as corresponding eigenfunction. When either $(d=3$, $\alpha\geqslant 0)$ or $(d\in\{2,3\}$, $\alpha=\infty)$ i.e.~in absence of eigenvalues, we set by convention $\lambda_{\alpha}=0$, so that in any case $-\lambda_{\alpha}=\inf\sigma(-\Delta_{\alpha})$.  

Given $\alpha\in\R$, $s\geqslant 0$, we introduce the scale of Banach spaces
\begin{equation}\label{eq:singular_spaces}H_{\alpha}^s(\R^d):=\operatorname{Dom}\big((1+\lambda_{\alpha}-\Delta_{\alpha})^{\frac s2}\big),\quad \|u\|_{H_{\alpha}^s}:=\|(1+\lambda_{\alpha}-\Delta_{\alpha})^{\frac s2}u\|_{L^2}.
\end{equation}
When $\alpha=\infty$, we recover the classical scale of Sobolev spaces $H^s(\R^d)$. Let us focus then on the regime $\alpha\in\R$, which corresponds to a non-trivial singular interaction at the origin. 
The case $s=2$ corresponds to the definition of the operator domain, namely $H^2_\alpha(\R^d)=\mathcal{D}(-\Delta_\alpha)$.\\
The case $s=1$ yields the domain of the quadratic form $\mathcal{Q}_{\alpha}$ associated to $-\Delta_{\alpha}$, and given $\omega>0$ we have the explicit characterization
\begin{align}\label{eq:form_domain}
H^1_{\alpha}(\R^d)&=\big\{u\in L^2(\R^d)\,:\,u=f+q\G_\omega\,\,f \in H^1(\R^d),\,q\in\R\big\},\\
\label{eq:quad_form1}
	\mathcal{Q}_{\alpha}(u)&=\int_{\R^d}\left(|\nabla f|^2+\omega|f|^2-\omega|u|^2\right)dx+\beta_{\alpha}(\omega)q^2.
\end{align}
Note that $H^1_{\alpha}(\R^d)$ is actually independent of $\alpha\in\R$, see also point (ii) of Proposition \ref{pr:hasp}. In other words, the domain of the quadratic form associated with $-\Delta_\alpha$ is the same for every $\alpha\in\R$, although the quadratic form itself depends on $\alpha$. Observe that elements in the form domain still admit a well-defined decomposition into a regular and a singular part, but unlike \eqref{domain} no relation between the regular component and the charge $q$ holds. In other words, the boundary condition at the singularity is lost. Finally, for $s\in(0,2]$, we define $H_{\alpha}^{-s}(\R^d)$ as the dual space of $H_{\alpha}^{s}(\R^d)$.

\subsection{The equivalence}
Associated with a Schr\"odinger operator $-\Delta_{\alpha}$ with point interaction, we consider the semilinear equation \eqref{eq:point_elliptic1}, here reproduced:
\begin{equation*}
	(-\Delta_{\alpha}+\lambda)u=\sigma|u|^{p-1}u.
\end{equation*}
As anticipated, our goal is to establish a rigorous connection between equations \eqref{eq:singular_elliptic} and \eqref{eq:point_elliptic1}. We will work in the following regime:
\begin{equation}\label{reg:pd}
\begin{cases}
p\in(1,\infty)&\mbox{if }d=2,\\
p\in(1,3)&\mbox{if }d=3.
\end{cases}
\end{equation}
Moreover, we impose the following a priori bound on solutions to \eqref{eq:singular_elliptic}: 
\begin{equation}\label{eq:below-green}\exists\;c,R>0\,:\,
\begin{cases}
\begin{aligned}
& -c \mathcal{G}_{\lambda}(x) \leqslant u(x)  && \forall x \in B_R^* && \text{if } \sigma = 1,\\
& -c \mathcal{G}_{\lambda}(x) \leqslant u(x) \leqslant c \mathcal{G}_{\lambda}(x) && \forall x \in B_R^* && \text{if } \sigma = -1.
\end{aligned}
\end{cases}
\end{equation}
Conditions \eqref{reg:pd} and \eqref{eq:below-green} guarantee that singular solutions to \eqref{eq:singular_elliptic} have leading local singularities proportional to the Green function $\mathcal{G}_{\lambda}$; we refer to Section \ref{isolatedsol} for further details, as well as for an overview of other type of singularities. More precisely, we are going to show that any solution $u$ to \eqref{eq:singular_elliptic} can be written as $u=f+q\mathcal{G}_{\lambda}$, where $f$ is the \emph{regular} component (i.e.~less singular than the Green kernel), and $q\in\R$ is given by the identity
\begin{equation}\label{eq:qlim}
q:=\lim_{x\to 0}\mathcal{G}_{\lambda}^{-1}u(x).
\end{equation}
This is exactly in line with the structure of point interactions described in Section \ref{subsec:Point}, with the admissible choice $\omega=\lambda$ for the auxiliary parameter $\omega>0$.

\begin{remark}
We point out that the a priori bounds \eqref{eq:below-green} will allow us to study also sign-changing solutions, extending classical results on isolated singularities, where the focus was only on positive solutions to \eqref{eq:singular_elliptic}.
\end{remark}

In what follows, we distinguish two qualitatively different regimes within the range of $p$ given by \eqref{reg:pd}.\smallskip

\underline{Trace Regime}: ($d=2,\, p>1$) or ($d=3,\,1<p<2$). In this regime, the regular component $f$ is continuous. In particular, there is a well-defined boundary condition relating $f(0)$ and $q$, which selects a unique value of $\alpha\in\R$ if $q\neq 0$. 

Furthermore, \eqref{eq:point_elliptic1} admits a variational characterization. Consider indeed, for any $\alpha\in\R$ and $\lambda>0$, the action functional 
\begin{equation}\label{eq:action}
S_{\lambda,\alpha}:H_{\alpha}^1(\R^d)\to\R,\quad S_{\lambda,\alpha}(u)={\textstyle\frac12}\left(\mathcal{Q}_{\alpha}(u)+\lambda\|u\|_{L^2}^2\right)-{\textstyle{\frac{\sigma}{p+1}}}\|u\|_{L^{p+1}}^{p+1},
\end{equation}
which is well-defined, since $H_{\alpha}^1(\R^d)\hookrightarrow L^{p+1}(\R^d)$ (as it follows by \eqref{lpgreen}, \eqref{eq:form_domain} and the standard Sobolev embedding), of class $\mathcal{C}^1$, and its critical points are exactly the solutions to \eqref{eq:point_elliptic1}, see e.g.~\cite[Proposition 4.2]{PoWa}. When $\alpha=\infty$, \eqref{eq:action} reduces to the standard nonlinear scalar-field action functional
\begin{equation}\label{eq:action_free}
S_{\lambda,\infty}:H^1(\R^d)\to\R,\quad S_{\lambda,\infty}(u)={\textstyle\frac12}\left(\|\nabla u\|^2_{L^2}+\lambda\|u\|_{L^2}^2\right)-{\textstyle{\frac{\sigma}{p+1}}}\|u\|_{L^{p+1}}^{p+1}.
\end{equation}

Our first main result shows the precise connection between solutions to  \eqref{eq:singular_elliptic}, solutions to \eqref{eq:point_elliptic1}, and critical points of $S_{\lambda,\alpha}$.
 
\begin{theorem}[Trace Regime]\label{th:main_equivalence}
Let $\lambda>0$, $\sigma=\pm 1$, and assume either $d=2$, $p\in(1,\infty)$, or $d=3$, $p\in(1,2)$. The following conditions are equivalent:
\begin{itemize}
\item[(i)] $u\in\mathcal{C}^2(\R^d\setminus\{0\})$, with $u(x)\to0$ as $|x|\to\infty$, satisfies \eqref{eq:below-green} and solves equation \eqref{eq:singular_elliptic};
\item[(ii)] there exist $\alpha\in\overline{\R}$ and $s>\frac{d}{2}$ such that $u\in H_{\alpha}^s(\R^d)$ and \eqref{eq:point_elliptic1} holds as an identity in $H_\alpha^{s-2}(\R^d)$;
\item[(iii)] there exists $\alpha\in \overline{\R}$ such that $u\in H^1_\alpha(\R^d)$ is a critical point of $S_{\lambda,\alpha}$.
\end{itemize}

If any of the above conditions is satisfied, then there exists $q\in\R$ determined by \eqref{eq:qlim}, and setting $f:=u-q\mathcal{G}_{\lambda}$ one has $f\in \mathcal{C}(\R^d)$. Moreover, $u$ extends to a $\mathcal{C}^2$ function on $\R^d$ if and only if $q=0$, in which case $u\in H^2(\R^d)$. If $q\neq0$, then $\alpha$ is uniquely determined by
\begin{equation}\label{eq:q-alpha}
\beta_{\alpha}(\lambda)q=f(0).
\end{equation}
\end{theorem}

The meaning of the above result is the following. In the trace regime, the set of singular solutions to \eqref{eq:singular_elliptic} satisfying \eqref{eq:below-green} can be parametrized by  $\alpha\in\mathbb{R}$, which determines the boundary condition \eqref{eq:q-alpha} (well-defined since the regular component is continuous). Once $\alpha$ is fixed, we obtain the operator-theoretic formulation \eqref{eq:point_elliptic1}, which in turn yields the variational characterization in terms of the action $S_{\lambda,\alpha}$.\par\smallskip

\underline{Free Regime}. In this regime, the regular component of a singular solution is not bounded at the origin, with a local singularity weaker than $\mathcal{G}_{\lambda}$. Thus, no boundary condition is selected, and $\alpha$ can be chosen arbitrarily in $\R$.

\begin{theorem}[Free Regime]
\label{th:equivalence-weak}
Let $\lambda>0$, $\sigma=\pm1$, $d=3$, and $p\in[2,3)$. The following conditions are equivalent:
\begin{itemize}
\item[(i)] $u\in\mathcal{C}^2(\R^3\setminus\{0\})$, with $u(x)\to0$ as $|x|\to\infty$, satisfies \eqref{eq:below-green}, and solves equation \eqref{eq:singular_elliptic};
\item[(ii)] there exists $\alpha\in\overline{\R}$ and  $s\in\left(\frac12,2\right]$ such that $u\in H^s_\alpha(\R^3)$, and \eqref{eq:point_elliptic1} holds as an identity in $H^{s-2}(\R^3)$.
\end{itemize}
If any of the above conditions is satisfied, then there exists $q\in\R$ determined by \eqref{eq:qlim}, and $u$ extends to a $\mathcal{C}^2$ function on $\R^3$ if and only if $q=0$, in which case $u\in H^2(\R^3)$. If $q\neq0$, then condition (ii) holds for any $\alpha\in\R$, $s\in(\frac{1}{2},\frac{7}{2}-p)$, and $f:=u-q\G_{\lambda}$ is singular at the origin, with 
\begin{equation}\label{subleading}
f(x)
\sim
\begin{cases}
\ln |x|,&p=2,\\
|x|^{2-p},&p\in(2,3),
\end{cases}
\qquad\mbox{as }\quad x\to0,
\end{equation}
\end{theorem}

The singular behavior in \eqref{subleading} corresponds to the first non-linear correction to $\G_{\lambda}$, formally given by $\sigma|q|^{p-1}q\G_{\lambda}\ast\G_{\lambda}^p$.

\begin{remark}
Let $s,\alpha$ be as in condition (ii) of Theorems \ref{th:main_equivalence} and \ref{th:equivalence-weak} (see Remark \ref{re:strong} for an explicit characterization in terms of $d,p$). Setting $\mathcal{N}(u):=|u|^{p-1}u$, we have $\mathcal{N}$ that maps $H_{\alpha}^s(\R^d)$ into $H^{s-2}(\R^d)$ in the whole range \eqref{reg:pd}. Moreover:
\begin{itemize}
\item In the Trace Regime, $H_{\alpha}^{s-2}(\R^d)\cong H^{s-2}(\R^d)$ (point (i) of Proposition \ref{pr:hasp}), so \eqref{eq:point_elliptic1} holds in an operatorial sense, namely as an identity in $H_{\alpha}^{s-2}(\R^d)$;
\item In the Free Regime, $H_{\alpha}^{s-2}(\R^3)\subsetneq H^{s-2}(\R^3)$ (point (ii) of Proposition \ref{pr:hasp}) and $\mathcal{N}$ does \emph{not} map $H_{\alpha}^s(\R^3)$ into $H_{\alpha}^{s-2}(\R^3)$, so \eqref{eq:point_elliptic1} can be interpreted only as an identity in $H^{s-2}(\R^3)$.
\end{itemize}
\end{remark}

\subsection{Variational structure and critical points of the action}\label{sec:variational}
An important consequence of Theorem \ref{th:main_equivalence} is that, in the Trace Regime,  every singular solution to \eqref{eq:singular_elliptic} is a critical point of  $S_{\lambda,\alpha}$ for a unique $\alpha\in\R$. From now on, we restrict our attention to the source case $\sigma=1$.

The variational analysis of $S_{\lambda,\alpha}$ has attracted growing attention in recent years. One of the main objects of interest is the class of \emph{action ground states}, namely action-minimizing critical points: more precisely, considering the set
\begin{equation*}
\mathsf{C}_{\lambda,\alpha}:=\big\{u \in H_\alpha^1\left(\mathbb{R}^d\right) \mid u \neq 0,\, S_{\lambda,\alpha}^{\prime}(u)=0\big\}.
\end{equation*}
of non-zero critical points of $S_{\lambda,\alpha}$, we define the set of (action) ground states as
\begin{equation*}
\mathsf{G}_{\lambda,\alpha}:=\big\{u \in \mathsf{C}_{\lambda,\alpha} \mid S_{\lambda,\alpha}(u) \leqslant S_{\lambda,\alpha}(v)\quad\forall\, v \in \mathsf{C}_{\lambda,\alpha}\big\}.
\end{equation*}
It is known \cite{ABCT-2d,ABCT-3d,FGI} that, for $\lambda>\lambda_{\alpha}$, $\mathsf{G}_{\lambda,\alpha}$ is non-empty, every ground state is singular and, up to sign (respectively, up to a constant phase in the complex-valued setting), positive and radially decreasing. These ground states give a family of positive singular solutions to \eqref{eq:singular_elliptic}.

Considering the singular, time-dependent nonlinear Schr\"odinger equation
\begin{equation*}
\ii\partial_t\psi=-\Delta_{\alpha}\psi-|\psi|^{p-1}\psi,
\end{equation*}
a ground state $u$ corresponds to the quasi-periodic solution $\psi(t,x)=e^{\ii\lambda t}u(x)$, whose stability/instability is addressed \cite{FGI, finco-noja}. The existence of action-minimizing solutions has been also investigated for singular Schr\"odinger equations with non-local nonlinearities \cite{GMS-CPDE,DPR-SN,DPR-H}, additional Coulomb interactions \cite{boni-gallone}, and Kirchhoff-type terms \cite{DPR-NA}.

Our next main result shows that, for $\lambda>\lambda_{\alpha}$, $S_{\lambda,\alpha}$ admits infinitely many singular critical points besides ground states, partially extending to the singular setting classical results available in the regular case -- see e.g.~\cite{BerLio} and references therein. 
\begin{theorem}\label{th:main_critical_points}
Fix $\sigma=1$, $\alpha\in\R$, $\lambda>\lambda_{\alpha}$, and $d=2$, $p\in(1,\infty)$ or $d=3$, $p\in(1,2)$. Then $S_{\lambda,\alpha}$ admits infinitely many singular, radial critical points.
\end{theorem}
The above result fits into the broader context of variational analysis for models with zero-range perturbations. Its proof relies on a general procedure developed by Ambrosetti and Rabinowitz in \cite{AR-73} -- see Section \ref{sec:var} for details.  In addition, as shown below, when combined with other structural properties of Schr\"odinger equations with point interactions, it can also provide new insights into elliptic PDEs with isolated singularities.

\begin{remark}
In Theorems~\ref{th:main_equivalence}-\ref{th:equivalence-weak}, $\lambda>0$ is the
datum of \eqref{eq:singular_elliptic}, while $\alpha$ is determined a posteriori by each
singular solution. More precisely, $\alpha$ is uniquely determined by \eqref{eq:q-alpha} in
the Trace Regime, while it is arbitrary in $\R$ in the Free Regime.
Theorem~\ref{th:main_critical_points} is instead an existence result obtained by
variational methods, which requires the functional $S_{\lambda,\alpha}$, hence both $\lambda$ and $\alpha$, to be fixed in advance. The assumption $\lambda>\lambda_\alpha$ (i.e.,
$\beta_\alpha(\lambda)>0$) ensures that the quadratic part of the action is positive definite. With $\lambda$ fixed, this condition equivalently reads $\alpha>\alpha_\lambda$,
where $\alpha_\lambda$ is determined by $\beta_{\alpha_\lambda}(\lambda)=0$. In view of
Theorem~\ref{th:main_equivalence}, for every $\lambda>0$
Theorem~\ref{th:main_critical_points} then yields infinitely many radial singular solutions to
\eqref{eq:singular_elliptic} for each $\alpha>\alpha_\lambda$, all satisfying $f(0)q>0$
by \eqref{eq:q-alpha}.
\end{remark}

\subsection{Positive versus nodal solutions in \texorpdfstring{$\R^2$}{R2}}
Let us restrict now our attention to dimension $d=2$, still focusing on the source case. It is proved in \cite{fukaya-uniqueness} that for any fixed $\alpha\in\R$, $\lambda>\lambda_{\alpha}$ and $p>1$, equation \eqref{eq:point_elliptic1} admits a unique positive, radial solution, which coincides with the (unique up to sign) action ground state -- see Proposition \ref{pr:uniqueness} for a complete statement. This extends the classical uniqueness result of \cite{Kwong} for positive, regular solutions to the singular setting, at least in two dimensions. In the regular case, uniqueness holds up to sign and in addition up to translations; the latter condition is not required in the singular case because the point interaction breaks the translational invariance.

The above information, combined with Theorems \ref{th:main_equivalence} and \ref{th:main_critical_points}, has two relevant consequences.

The first is a structure theorem for positive solutions to \eqref{eq:singular_elliptic}, whose proof also relies on a rigidity result ensuring that positive solutions to \eqref{eq:singular_elliptic} are necessarily radially symmetric -- see Proposition \ref{pr:radial}.

\begin{theorem}\label{th:structure}
Fix $d=2$, $\sigma=1$, $\lambda>0$, $p>1$, and let $u\in\mathcal C^2(\R^2\setminus\{0\})$ be a positive solution to equation \eqref{eq:singular_elliptic} such that $u(x)\to0$ as $|x|\to\infty$.
\begin{itemize}
	\item[(i)] If $u$ is singular, then there exists $\alpha\in\R$ such that 
	$\mathsf{G}_{\lambda,\alpha}=\{\pm u\}$.
	\item[(ii)] If $u$ is regular, then
	$\mathsf{G}_{\lambda,\infty}=\{\pm u(\cdot-x_0)\,:\,x_0\in\R^2\}$.
\end{itemize}
\end{theorem}
Equivalently, the above result shows that $u$ coincides with the unique (up to sign, and additionally up to translations in the regular case) ground state of $S_{\lambda,\alpha}$ for some $\alpha\in\overline{\R}$, with $\alpha=\infty$ if and only if $u$ is regular.

The second consequence is the existence of \emph{nodal}, singular solutions to the semilinear elliptic equation \eqref{eq:singular_elliptic}.

\begin{theorem}\label{th:nodal}
Fix $d=2$, $\sigma=1$, $\lambda>0$ and $p>1$. There exist infinitely many singular, radial, nodal solutions to \eqref{eq:singular_elliptic} vanishing at infinity.
\end{theorem}

As already observed in the introduction, the above result appears to be new. 
We mention that nodal solutions to the Schr\"odinger equation with point interaction \eqref{eq:point_elliptic1} are constructed in the recent preprint \cite{DPR-nodal} via a Lyapunov--Schmidt reduction, but only for a significantly more restricted range of the parameters $p$ and $\lambda$ than the one considered here.

The proof of Theorem \ref{th:nodal} relies on two key ingredients: the existence of infinitely many singular critical points of the action functional $S_{\lambda,\alpha}$, and the uniqueness of positive solutions to \eqref{eq:point_elliptic1}. In the regular setting (standard Laplacian and no point interaction), besides the above strategy based on multiplicity--uniqueness results, various alternative approaches have been developed in the literature, which also often provide refined structural information on the nodal sets. Extending these approaches to the singular case is an interesting direction for future research.

\begin{remark}
	We expect that statements analogous to Theorems \ref{th:structure} and \ref{th:nodal} also hold in dimension $d=3$. However, since uniqueness of positive solutions to \eqref{eq:point_elliptic1} is still unknown in that case, we cannot apply the strategy developed here. In this perspective, we point out that some arguments in \cite{fukaya-uniqueness} rely crucially on working in dimension two, and therefore new ideas are required.
\end{remark}

\subsection{Comparison with classical theory of singular solutions}\label{isolatedsol}
We conclude with an overview of classical results on isolated singularities for \eqref{eq:singular_elliptic} in dimension $d>1$, based on the monograph \cite{Veron-book} and references therein. This highlights the role of the assumptions required for the operator-theoretic approach provided by point interactions.

A \emph{weak} (or Green-type) isolated singularity is a solution with
\begin{equation*}
u(x)\sim q\G_\lambda(x)\mbox{ as }x\to 0,\mbox{ for some }
        q\neq 0.
\end{equation*}
These solutions exist if and only if $\G_{\lambda}^p\in L^1_{\mathrm{ loc}}(\R^d)$, namely for
\begin{equation}\label{eq:sub}
\begin{cases}
p\in(1,\infty)&\mbox{if }d=2,\\
p\in(1,\frac{d}{d-2})&\mbox{if }d\geqslant 3,
\end{cases}
\end{equation}
which justify the choice of the regime \eqref{reg:pd} (as already discussed in the introduction, the restriction $d\leqslant 3$ in this paper is related instead to the existence of non-trivial point interactions).

Moreover, if $u$ is a weak-type singularity, then it satisfies
\begin{equation*}
(-\Delta+\lambda)u=\sigma|u|^{p-1}u+q\delta_0
\end{equation*}
in the sense of distribution. The above Dirac mass correction is compatible with the action of point interaction given by \eqref{eq:dirac_elliptic}, and will be exploited in the proofs of Theorems \ref{th:main_equivalence}-\ref{th:equivalence-weak}.

Besides the branch of weak-type solutions, equation \eqref{eq:singular_elliptic} also admits other kinds of isolated singularities for suitable values of $d$, $\sigma$ and $p$. Let us discuss separately the absorption and source cases.\par\smallskip

\underline{Absorption case: $\sigma=-1$.} In the regime \eqref{eq:sub}, there exist also the so-called strong-type isolated singularities, exhibiting the power behavior
\begin{equation*}
u(x)\sim |x|^{-2/(p-1)}\mbox{ as }x\to 0\ .
\end{equation*} As a consequence, the two-sided Green bound near the
origin, imposed in \eqref{eq:below-green}, is necessary. When $d\geqslant 3$ and $p\geqslant\frac{d}{d-2}$, there are no positive, singular solutions to \eqref{eq:singular_elliptic}. We refer to Section \ref{ssa} for further discussion on the absorption case.\par\smallskip

\underline{Source case: $\sigma=1$.} In the regime \eqref{eq:sub}, the a priori lower bound $u\gtrsim -\G_{\lambda}$ near the origin we imposed in \eqref{eq:below-green} guarantees that the only possible isolated singularities are of weak-type. To the best of our knowledge, it is unknown whether this lower bound is actually necessary; still it emerges naturally in our approach, which exploits a classical result of Brezis-Lions, see Proposition \ref{lem:Brezis-Lions}. 

In the regime $d\geqslant 3$, $\frac{d}{d-2}\leqslant p<\frac{d+2}{d-2}$, there exists other isolated singularities, whose local behavior at $x=0$ is given by
$$
\begin{cases}
|x|^{-(d-2)}\big|\log{|x|}\big|^{(2-d)/2}&p=\frac{d}{d-2},\\
|x|^{-2/(p-1)}&\frac{d}{d-2}<p<\frac{d+2}{d-2}.
\end{cases}
$$
Finally, when $d\geqslant 3$ and $p>\frac{d+2}{d-2}$, there are neither singular nor regular solutions.

Summarizing, the description by means of point interactions is not a classification of all isolated singularities, but rather an operator-theoretic characterization of the weak-type solutions in dimensions $d=2,3$. As we have shown, this framework sheds new light on the study of singular solutions, and it would therefore be interesting to extend it to encompass the case $d\geqslant 4$, as well as others type of isolated singularities.



\section{Proofs of the main results}\label{sec:Proofs}
\subsection{Equivalence} We prove here our first two main results, Theorem \ref{th:main_equivalence} and Theorem \ref{th:equivalence-weak}, on the connection between equations \eqref{eq:singular_elliptic} and \eqref{eq:point_elliptic1}.  The proofs are based upon the combination of three main tools:
\begin{itemize}
	\item[-] a suitable characterization of the Sobolev spaces $H_{\alpha}^s(\R^d)$ adapted to the point interaction, see Proposition \ref{pr:hasp} and Remark \ref{re:transition};
	\item[-] explicit formulas for the action of $-\Delta_{\alpha}$ at fractional Sobolev regularity, see Lemma \ref{le:csdi};
	\item[-] a refined version of the Brezis-Lions Lemma for linear elliptic equations, see Proposition \ref{lem:Brezis-Lions}.
\end{itemize} 

We start our analysis by providing an explicit characterization of the Sobolev spaces $H_{\alpha}^s(\R^d)$ adapted to $-\Delta_{\alpha}$, defined by \eqref{eq:singular_spaces}, in the relevant regime $s\in(0,2]$. The following proposition is proved in \cite{GeoRas-full}, which also deal with the more general case of $H_{\alpha}^{s,p}$-spaces with $p\neq 2$ (see also \cite{Georgiev-M-Scandone-2016-2017,MOS,GMS-CPDE,GeoRas-pro,GeoRas-2D} for related results).

\begin{proposition}\label{pr:hasp}
Let $d=2,3$ and $\alpha\in\R$. We distinguish three regimes:
\begin{itemize}
	\item[(i)] \underline{No Singular Component}: $d=2$, $s\in(0,1)$ or $d=3$, $s\in(0,\frac12)$. We have \begin{equation*}H_{\alpha}^s(\R^d)\cong H^s(\R^d)\end{equation*}
	 as an equivalence between Banach spaces;
	\item[(ii)] \underline{Singular Component, Free regime}: $d=2$, $s=1$ or $d=3$, $s\in(\frac12,\frac32)$.\\[0.2em] Given $\omega>0$, we have
\begin{gather*}H_{\alpha}^s(\R^d)=\{u\in L^2(\R^d)\,:\,u=f+q\mathcal{G}_{\omega},\, f\in H^s(\R^d),\,q\in\R\},\\
\|f+q\mathcal{G}_{\omega}\|_{H_{\alpha}^s}\approx \|f\|_{H^s} +|q|;
\end{gather*}
	\item[(iii)] \underline{Singular Component, Trace Regime}: $d=2$, $s\in(1,2]$ or $d=3$, $s\in(\frac32,2]$. \\[0.2em] Given $\omega>0$, we have
	\begin{gather*}
	H_{\alpha}^s(\R^d)=\{u\in L^2(\R^d)\,:\,u=f+q\mathcal{G}_{\omega},\, f\in H^s(\R^d)\,,\,\beta_{\alpha}(\omega)q=f(0)\},\\
	\|f+q\mathcal{G}_{\omega}\|_{H_{\alpha}^s}\approx\|f\|_{H^s}.
	\end{gather*}
\end{itemize}
\end{proposition}

\begin{remark}\label{re:transition}
The characterization of $H_{\alpha}^{s}(\R^3)$ in the threshold cases $s=\frac12,\frac32$ is more involved \cite{Georgiev-M-Scandone-2016-2017}, and it is not needed for our purposes. We present here only some useful partial information: given $\omega>0$,\par\smallskip\noindent
(i) both $H^{1/2}(\R^3)$ and $\operatorname{span}{\{\G_{\omega}\}}$ are closed subspaces of $H_{\alpha}^{1/2}(\R^3)$; \par\smallskip\noindent
(ii) there results that 
\begin{equation}\label{tremezzi}
H^{ 3/2}_{\alpha}(\R^3)=\{u\in L^2(\R^3)\,:\,u=f+q\mathcal{G}_{\omega},\, f\in\widetilde{H}^{3/2}(\R^3)\,,\,\beta_{\alpha}(\omega)q=f(0)\},
\end{equation}
where the space $\widetilde{H}^{3/2}(\R^3)$ satisfies the following chain of inclusions
\begin{equation*}
H^{3/2+}(\R^3)\subseteq\widetilde{H}^{3/2}(\R^3)\subseteq H^{3/2}(\R^3)\cap\mathcal{C}(\R^3).
\end{equation*}
\end{remark}

Next, we study the action of $-\Delta_{\alpha}$ in the regime covered by point (ii)-(iii) of Proposition \ref{pr:hasp}, namely when the adapted Sobolev space $H_{\alpha}^s(\R^d)$ decouples regular and singular (i.e.~proportional to the Green kernel) components.

\begin{lemma}\label{le:csdi}
Let $d=2,3$, $\alpha\in\R$, and $s\in[1,2]$ when $d=2$, $s\in(\frac12,2]$ when $d=3$. Fix $\omega>0$. Then both $H^{2-s}(\R^d)$ and $\operatorname{span}\{\mathcal{G}_{\omega}\}$ are closed subspaces of $H_{\alpha}^{2-s}(\R^d)$. If $u=f+q\mathcal{G}_{\omega}\in H^s_{\alpha}(\R^d)$ then
\begin{itemize}
	\item[(i)] $-\Delta_{\alpha}u$ can be identified with a distribution in $H^{s-2}(\R^d)$, satisfying
	\begin{equation}\label{eq:distrib_id}
		(-\Delta_{\alpha}+\omega)u=(-\Delta+\omega)u-q\delta=(-\Delta+\omega)f.
	\end{equation}
\item[(ii)] $-\Delta_{\alpha}u$ acts on the Green function $\mathcal{G_{\omega}}$ through the identity
\begin{equation}\label{eq:test_green}
\big\langle(-\Delta_{\alpha}+\omega)u,\mathcal{G}_{\omega}\big\rangle_{H_{\alpha}^{s-2}\times H_{\alpha}^{2-s}}=\beta_{\alpha}(\omega)q.
\end{equation}
\end{itemize}
\end{lemma}
\begin{proof}
The fact that $H^{2-s}(\R^d)$ and $\operatorname{span}\{\mathcal{G}_{\omega}\}$ are closed subspaces of $H_{\alpha}^{2-s}(\R^d)$ directly follows from the characterization of the adapted Sobolev spaces provided in Proposition \ref{pr:hasp} and Remark \ref{re:transition}. Moreover, $-\Delta_{\alpha}u$ defines a continuous functional on $H_{\alpha}^{2-s}(\R^d)$, whence also on $H^{2-s}(\R^d)$ and $\operatorname{span}\{\mathcal{G}_{\omega}\}$. As a consequence, we have:\par\smallskip
\textbf{(i)} $-\Delta_{\alpha}u$ can be identified with a distribution in $H^{s-2}(\R^d)$. Moreover, if we fix $\chi\in\mathcal{C}^{\infty}_0(\R^d\setminus\{0\})\subseteq\mathcal{D}(-\Delta_{\alpha})$, by \eqref{action}, we get
\begin{equation}\label{id:dense}
	\begin{split}
\langle (-\Delta_{\alpha}+\omega)u,\chi\rangle&=\langle f+q\mathcal{G}_{\omega}, (-\Delta_{\alpha}+\omega)\chi\rangle=\langle f+q\mathcal{G}_{\omega}, (-\Delta+\omega)\chi\rangle\\
&=\langle (-\Delta+\omega)(f+q\mathcal{G}_{\omega}), \chi\rangle
=\langle (-\Delta+\omega)f+q\delta, \chi\rangle\\
&=\langle (-\Delta+\omega)f, \chi\rangle.
	\end{split}
\end{equation}
Since $\mathcal{C}^{\infty}_0(\R^d\setminus\{0\})$ is dense in $H^{2-s}(\R^d)$, we deduce from \eqref{id:dense} that 
\begin{equation}\label{eq:dpd}
(-\Delta_{\alpha}+\omega)u=(-\Delta+\omega)f
\end{equation}
holds as an identity in $H^{s-2}(\R^d)$. Moreover
\begin{equation}
(-\Delta+\omega)u=(-\Delta+\omega)(f+q\mathcal{G}_{\omega})=(-\Delta_{\alpha}+\omega)u+q\delta,
\end{equation}
which combined with \eqref{eq:dpd} proves identity \eqref{eq:distrib_id}.\par\smallskip
\textbf{(ii)} Let $\{\varphi_n\}_{n\in\N}$ be a sequence of smooth, compactly supported functions on $\R^d$, with $\varphi_n(0)=\beta_{\alpha}(\omega)$ $\forall n\in\N$, such that $\varphi_n\to 0$ in $H^{2-s}(\R^d)$. Such a sequence exists because point evaluation is not continuous on $H^{2-s}(\R^d)$ in the range $2-s\leqslant\frac{d}{2}$. We set also
\begin{equation*}
\mathcal{G}_{\omega}^{(n)}:=\varphi_n+\mathcal{G}_{\omega}.
\end{equation*}
Owing to \eqref{domain}, $\mathcal{G}_{\omega}^{(n)}\in H^2_{\alpha}(\R^d)$ $\forall\,n\in\N$, and by virtue of point (ii) of Proposition \ref{pr:hasp} we have $\mathcal{G}_{\omega}^{(n)}\to\mathcal{G}_{\omega}$ in $H_{\alpha}^{2-s}(\R^d).$ Then we obtain
\begin{equation*}
\begin{split}
\big\langle(-\Delta_{\alpha}+\omega)u,\mathcal{G}_{\omega}\big\rangle_{H_{\alpha}^{s-2}\times H_{\alpha}^{2-s}}&=\lim_{n\to\infty}\big\langle(-\Delta_{\alpha}+\omega)u,\mathcal{G}^{(n)}_{\omega}\big\rangle_{H_{\alpha}^{s-2}\times H_{\alpha}^{2-s}}\\
&=\lim_{n\to\infty}\big\langle f+q\mathcal{G}_{\omega},(-\Delta_{\alpha}+\omega)(\varphi_n+\mathcal{G}_{\omega})\big\rangle_{L^2\times L^2}\\
&=\lim_{n\to\infty}\big\langle f+q\mathcal{G}_{\omega},(-\Delta+\omega)\varphi_n\big\rangle_{L^2\times L^2}\\
&=\lim_{n\to\infty}\big\langle (-\Delta+\omega)f,\varphi_n\big\rangle_{H^{s-2}\times H^{2-s}}\\
&+\lim_{n\to\infty} q\big\langle \mathcal{G}_{\omega},(-\Delta+\omega)\varphi_n\big\rangle_{L^2\times L^2}\\
&=\lim_{n\to\infty}q\varphi_n(0)=q\beta_{\alpha}(\omega),
\end{split}
\end{equation*}
which proves identity \eqref{eq:test_green} and concludes the proof.
\end{proof}

Finally, we recall the Brezis-Lions Lemma \cite{Brezis-Lions-Lemma}, which provides suitable structural properties of super-solutions to linear, inhomogeneous elliptic equations.

\begin{proposition}[Brezis-Lions Lemma]
\label{lem:Brezis-Lions}
Let $d\geqslant 2$, $\lambda>0$ and $u\in L^{1}_{\mathrm{loc}}(\R^{d}\setminus\{0\})$ satisfy the following properties:
\begin{itemize}
\item[(i)] $\Delta u \in L^{1}_{\mathrm{loc}}(\R^{d}\setminus\{0\})$, in the sense of distributions on $\R^d\setminus\{0\}$;
  \item[(ii)]  $(-\Delta+\lambda) u(x)\geqslant g(x)$ a.e.~in $\R^{d}$ for some $g\in L^{1}_{\mathrm{loc}}(\R^{d})$; 
 \item[(iii)] there exists $R>0$ and $c\geqslant 0$ such that 
 \begin{equation*}
 u\geqslant -c\mathcal{G}_{\lambda}\quad \text{a.e.~in }B_R.  
 \end{equation*}
\end{itemize}
 Then $u\in L^{1}_{\mathrm{loc}}(\R^{d})$, and there exist $q\in\R,\varphi\in L^1_{\mathrm{loc}}(\R^d)$ such that 
 \begin{equation}
 \label{eq:distr-form-delta}
-\Delta u=\varphi+q\delta\quad \text{in}\quad \mathcal{D}'(\R^{d}).
 \end{equation}
\end{proposition}

In the seminal paper \cite{Brezis-Lions-Lemma}, the authors require $u\geqslant 0$, though they suggest that the non-negativity assumption can be actually replaced by condition (iii) above, which is crucial for our application to nodal solutions of \eqref{eq:point_elliptic1} in the source case. We provide an explicit proof of Proposition \ref{lem:Brezis-Lions} in Appendix \ref{app}.

\begin{remark} 
If condition (iii) is dropped, $u$ may still be in $L^1_{\mathrm{loc}}(\R^d)$, but \eqref{eq:distr-form-delta} would generally involve an additional distributional correction of order one. More precisely, allowing $u$ to have a singularity of order $-\mathcal{G}_{\lambda}'\in L^1_{\mathrm{loc}}(\R^d)$, one would get a $\delta'$-term in \eqref{eq:distr-form-delta}.
\end{remark}

We are now able to prove an intermediate statement on the equivalence between equations \eqref{eq:singular_elliptic} and \eqref{eq:point_elliptic1}, that covers both the Trace and the Free Regime.

\begin{proposition}
\label{prop:equivalence}
Let $\lambda>0$, $\sigma=\pm 1$, and either $d=2$, $p\in(1,\infty)$ or $d=3$, $p\in(1,3)$. The following conditions are equivalent:
\begin{itemize}
\item[(i)] $u\in\mathcal{C}^2(\R^d\setminus\{0\})$, with $|u(x)|\to 0$ as $|x|\to\infty$, satisfies \eqref{eq:below-green} and solves equation \eqref{eq:singular_elliptic};
\item[(ii)] there exists $\alpha\in\overline{\R}$, and $s\in[1,2]$ when $d=2$ or $s\in(\frac12,2]$ when $d=3$, such that $u\in H_{\alpha}^s(\R^d)$ and satisfies \eqref{eq:point_elliptic1} as an identity in $H^{s-2}(\R^d)$.
\end{itemize}
If any of the above conditions is satisfied, then there exists $q\in\R$ determined by \eqref{eq:qlim},
and $u$ extends to a  $\mathcal{C}^2$ function on $\R^d$ if and only if $q=0$, in which case $u\in H^2(\R^d)$.
\end{proposition}
\begin{proof}
We divide the proof into three steps.\par\smallskip
\emph{Step 1: (i)$\Rightarrow$(ii).} Assume that (i) holds. In particular, $u\in L^1_{\mathrm{loc}}(\R^d\setminus\{0\})$ and
\begin{equation*}\Delta u=\lambda u-\sigma|u|^{p-1}u \in L^1_{\mathrm{loc}}(\R^d\setminus\{0\}).\end{equation*} Moreover, the condition $\sigma u\geqslant -c\mathcal{G}_{\lambda}$ on $B^*_R$, valid for $\sigma=\pm 1$ in view of \eqref{eq:below-green}, implies
\begin{equation*}(-\Delta+\lambda)u=\sigma|u|^{p-1}u\geqslant  -c^p\mathcal{G}_{\lambda}^p\quad\mbox{on }B^*_R.\end{equation*}
Hence, setting 
\begin{equation*}
g(x)=\begin{cases}
-c^p\mathcal{G}_{\lambda}^p&x\in B_R^*\\
-|u|^p&x\not\in B_R^*,
\end{cases}
\end{equation*}
we have $g\in L^1_{\mathrm{loc}}(\R^d)$ and $(-\Delta+\lambda) u(x)\geqslant g(x)$ a.e.~in $\R^{d}$. We can then apply Proposition \ref{lem:Brezis-Lions}, which yields $u\in L^1_{\mathrm{loc}}(\R^d)$ and the identity
\begin{equation*}-\Delta u=\varphi + q\delta\quad\mbox{in }\mathcal{D}'(\R^d),\end{equation*}
for some $q\in\R$ and $\varphi\in L_{\mathrm{loc}}^1(\R^d)$. Since \eqref{eq:singular_elliptic} is satisfied, we necessarily have $|u|^{p-1}u\in L^1_{\mathrm{loc}}(\R^d)$ and $\varphi=-\lambda u+\sigma|u|^{p-1}u$, namely
\begin{equation}\label{eq:equ}
(-\Delta+\lambda)u=\sigma|u|^{p-1}u+q\delta
\end{equation}
holds as an identity in $\mathcal{D}'(\R^d)$. 

Since $u(x)\to 0$ as $|x|\to\infty$, there exists $R_0>0$ such that
$|u|^{p-1}\leqslant\frac{\lambda}{2}$ in $\mathbb R^d\setminus B_{R_0}$. Moreover, by Kato's
inequality,
\begin{equation*}
(-\Delta+\lambda)|u|\leqslant \sigma |u|^p
\end{equation*}
in the sense of distributions outside $B_{R_0}$. Hence, in the source case
$\sigma=1$, and a fortiori in the absorption case $\sigma=-1$, one has
\begin{equation*}
\big(-\Delta+{\textstyle{\frac{\lambda}{2}}}\big)|u|\leqslant 0
\qquad\text{in }\mathbb R^d\setminus B_{R_0}.
\end{equation*}
Comparing $|u|$ in exterior annuli with
$C\mathcal G_{\lambda/2}+\varepsilon$, and then letting the outer radius tend to
infinity and $\varepsilon\downarrow 0$, we obtain
\begin{equation*}
|u(x)|\leqslant C\mathcal G_{\lambda/2}(x)
\qquad \text{for } |x|\geqslant R_0.
\end{equation*}
Hence $u$ decays exponentially to zero at infinity, and in particular both $u$ and $|u|^{p-1}u$ belong to $L^1(\mathbb R^d)$. Then, by \eqref{eq:equ}, we can write
\begin{equation}\label{eq:decf}
u=f+q\mathcal{G}_{\lambda},
\end{equation}
where $f\in (-\Delta+\lambda)^{-1}L^{1}(\R^d)$ satisfies 
\begin{equation}\label{eq:ellif}
(-\Delta+\lambda)f=\sigma|u|^{p-1}u.
\end{equation}
As a consequence of \eqref{eq:decf}, \eqref{eq:ellif} and \eqref{lpgreen}, $u$ belongs to $L^r(\R^d)$ for every $r\in[1,\infty)$ when $d=2$, and for every $r\in[1,3)$ when $d=3$. Thus we get
\begin{equation}\label{lpup}
|u|^{p-1}u\in\begin{cases}
L^{\infty-}(\R^d)&d=2\\
L^{\frac{3}{p}-}(\R^d)&d=3.
\end{cases}
\end{equation}
Owing to \eqref{eq:ellif}, \eqref{lpup}, Calderon-Zygmund estimates for Bessel potentials \cite[Chapter V, \S3]{Stein} and Sobolev embedding, the regularity of $f$ is then improved to
\begin{equation}\label{eq:improved_regularity}
f\in\begin{cases}
	H^2(\R^d)&d=2,\,p\in(1,\infty)\mbox{ or }d=3,\,p\in(1,\frac32)\\
	H^{(\frac72 - p)-}(\R^d)&d=3,\,p\in[\frac32,3).
\end{cases}
\end{equation}
In particular, $f\in H^s(\R^d)$ for $s=1$ when $d=2$ or $s\in(\frac12,\frac32)$ when $d=3$ (note indeed that $\frac{7}{2}-p>\frac{1}{2}$ for every $p<3$). This is exactly the range of $s$ covered by point (ii) of Proposition \ref{pr:hasp}, where the regular part is well-defined but no boundary condition appears. In particular, with the above choice of $s$, and taking $\alpha\in\R$ arbitrary, we obtain $u=f+q\mathcal{G}_{\lambda}\in H_{\alpha}^s(\R^d)$. Moreover, in view of Lemma \ref{le:csdi}, equation \eqref{eq:equ} implies that $u$ satisfies \eqref{eq:point_elliptic1} as an identity in $H^{s-2}(\R^d)$, thus proving (ii). \par\smallskip

\emph{Step 2: (ii)$\Rightarrow$(i).} Suppose that condition (ii) holds. Let $u=f+q\mathcal{G}_{\lambda}$, for some $q\in\R$ and $f\in H^s(\R^d)$ with $s\in[1,2]$ when $d=2$, $s\in(\frac12,2]$ when $d=3$. In view of Lemma \ref{le:csdi}, identity \eqref{eq:point_elliptic1} can be rewritten, in the sense of distributions on $\R^d$, as
\begin{equation*}
(-\Delta+\lambda)u=\sigma|u|^{p-1}u+q\delta.
\end{equation*}

In particular, $u$ satisfies equation \eqref{eq:singular_elliptic} on $\R^d\setminus\{0\}$, and elliptic regularity yields $u\in\mathcal{C}^2(\R^d\setminus\{0\})\cap H^{\frac{d}{2}+}(\R^d\setminus B_1)$, which entails also, by Sobolev-Morrey embedding, that $u$ is uniformly continuous on $\R^d\setminus B_1$. This readily implies $u(x)\to 0$ as $|x|\to +\infty$. Even though the proof of this fact is standard, we report it here for the sake of completeness. 

Suppose by contradiction that $u(x)$ does not tend to zero as $|x|\to\infty$.
Then there exist $c>0$ and a sequence $(x_j)$ with $|x_j|\to\infty$ such that
$|u(x_j)|\geqslant c$ for every $j$. By uniform continuity there exists $r>0$,
independent of $j$, such that
\begin{equation*}
	|u(x)|\geqslant\frac{c}{2}
	\qquad\text{for every }x\in B_r(x_j)\text{ and every }j.
\end{equation*}
Eventually passing to a subsequence, we may assume that the balls $B_r(x_j)$ are pairwise
disjoint. Consequently,
\begin{equation*}
	\int_{\R^d}|u|^2\,dx
	\geqslant\sum_{j=1}^{\infty}\int_{B_r(x_j)}|u|^2\,dx
	\geqslant\frac{c^2}{4}\sum_{j=1}^{\infty}|B_r|
	=+\infty,
\end{equation*}
contradicting $u\in L^2(\R^d)$.  Hence $u(x)\to0$ as $|x|\to\infty$.

We are left to show the bound \eqref{eq:below-green}. For later purposes, we are going to prove a stronger fact, namely
\begin{equation}\label{limfzero}
\lim_{x\to 0}\mathcal{G}_{\lambda}^{-1}f(x)=0.
\end{equation}
To this aim, we start by observing that \eqref{eq:equ} and \eqref{eq:ellif} are equivalent as identities in $\mathcal{D}'(\R^d)$, thus arguing as before we get \eqref{eq:improved_regularity}  (note that this may provide a better integrability for $f$ than originally assumed). In particular, when $d=2$, $p\in(1,\infty)$ or $d=3$, $p\in(1,2)$, we have $f\in L^{\infty}(\R^d)$, and \eqref{limfzero} follows by the singular behavior of $\mathcal{G}_{\lambda}$ at the origin. When instead $d=3$, $p\in[2,3)$, let us show that we can write
\begin{equation}\label{f_dec_rr}
f=f_{\mathrm{reg}}+f_{\mathrm{rad}},\quad f_{\mathrm{reg}}\in L^{\infty}(\R^3),\quad f_{\mathrm{rad}}\in H_{\mathrm{rad}}^{(\frac72-p)-}(\R^3).
\end{equation}
To this aim, we use a bootstrap procedure inspired by the approach developed in \cite{Naito-Sato}. Let us set $f_{\mathrm{reg}}^{(0)}=f$, $f_{\mathrm{rad}}^{(0)}=0$, and for $k\geqslant 1$ we inductively construct $f_{\operatorname{reg}}^{(k)}$ and $f_{\operatorname{rad}}^{(k)}$ as follows:
\begin{equation}\label{def_reg_rad}
f_{\mathrm{rad}}^{(k)}=\mathcal{G}_{\lambda}*\sigma\Big(|f_{\mathrm{rad}}^{(k-1)}+q\mathcal{G}_{\lambda}|^{p-1}(f_{\mathrm{rad}}^{(k-1)}+q\mathcal{G}_{\lambda})\Big),\qquad f_{\mathrm{reg}}^{(k)}=f-f_{\mathrm{rad}}^{(k)}.
\end{equation}
Note that for every $k\geqslant 0$, $f_{\mathrm{rad}}^{(k)}$ is radial and moreover
 \begin{equation}\label{eq:fH}
 f_{\mathrm{rad}}^{(k)},\;f_{\mathrm{reg}}^{(k)}\in H^{(\frac72-p)-}(\R^3)\hookrightarrow L^{3+}(\R^3).
 \end{equation}
This is clear for $k=0$ in view of \eqref{eq:improved_regularity}, and follows inductively for every $k\geqslant 1$ from the mapping properties of Bessel potentials, the bound \eqref{lpgreen} and the radial symmetry of $\mathcal{G}_{\lambda}$. Let us show that the integrability of $f_{\mathrm{reg}}^{(k)}$ actually improves at each step, eventually reaching $L^{\infty}$. To start with, we use \eqref{eq:ellif} and  \eqref{def_reg_rad} to obtain
\begin{equation}\label{eq:stica}
	\begin{split}
	\sigma(-\Delta+\lambda)f_{\operatorname{reg}}^{(k)}&=\big|f_{\operatorname{reg}}^{(k-1)}+f_{\operatorname{rad}}^{(k-1)}+q\mathcal{G}_{\lambda}\big|^{p-1}\big(f_{\operatorname{reg}}^{(k-1)}+f_{\operatorname{rad}}^{(k-1)}+q\mathcal{G}_{\lambda}\big)\\
&-\big|f_{\operatorname{rad}}^{(k-1)}+q\mathcal{G}_{\lambda}\big|^{p-1}\big(f_{\operatorname{rad}}^{(k-1)}+q\mathcal{G}_{\lambda}\big).
\end{split}
\end{equation}
The above identity and Lipschitz estimates for the map $x\mapsto |x|^{p-1}x$ thus yield
\begin{equation}\label{eq:slpit_fk}
\big|	(-\Delta+\lambda)f_{\operatorname{reg}}^{(k)}\big|\lesssim \Big(\big|f_{\operatorname{reg}}^{(k-1)}\big|^{p-1}+\big|f_{\operatorname{rad}}^{(k-1)}+q\mathcal{G}_{\lambda}\big|^{p-1}\Big)\big| f_{\operatorname{reg}}^{(k-1)}\big|.
\end{equation}
If $f_{\mathrm{reg}}^{(k-1)}\in L^{(\theta_{k-1})-}(\R^3)$, with $\theta_{k-1}\in(3,\infty)$, then setting 
\begin{equation*}\frac{1}{\om_{k-1}}=\frac{p-1}{3}+\frac{1}{\theta_{k-1}}\end{equation*}
we obtain 
\begin{equation}\label{reg_rhs}
\Big(\big|f_{\mathrm{reg}}^{(k-1)}\big|^{p-1}+\big|f_{\mathrm{rad}}^{(k-1)}+q\mathcal{G}_{\lambda}\big|^{p-1}\Big)f_{\mathrm{reg}}^{(k-1)}\in L^{(\om_{k-1})-}(\R^3),
\end{equation}
where we used H\"older inequality and exploited that $\mathcal{G}_{\lambda}\in L^{3-}(\R^3)$. Hence, setting
\begin{equation}\label{eq:theta_k}
\frac{1}{\theta_k}=\frac{1}{\omega_{k-1}}+\frac{1}{3}-1=\frac{p-3}{3}+\frac{1}{\theta_{k-1}},
\end{equation}
we deduce by \eqref{eq:slpit_fk}, \eqref{reg_rhs} and Hardy-Littlewood-Sobolev inequality that
\begin{equation}\label{eq:f_k}
f_{\mathrm{reg}}^{(k)}\in\begin{cases}
	L^{\theta_k-}(\R^3)&\mbox{if }\theta_k\geqslant 0\\
	L^{\infty}(\R^3)&\mbox{if }\theta_k<0.
\end{cases}
\end{equation}
Choosing then $K$ sufficiently large so that
\begin{equation*}\frac{1}{\theta_0}+K\left(\frac{p-3}{3}\right)<0,\end{equation*}
we get $f_{\mathrm{reg}}^{(K)}\in L^{\infty}(\R^3)$, as desired. The decomposition \eqref{f_dec_rr} then follows by setting $f_{\mathrm{reg}}:=f_{\mathrm{reg}}^{(K)}$, $f_{\mathrm{rad}}:=f_{\mathrm{rad}}^{(K)}$.
Using weighted Sobolev embedding for radial functions, see e.g.~\cite[Proposition 1]{cho-ozawa}, we deduce
\begin{equation*}|f_{\mathrm{rad}}(x)|\lesssim |x|^{-\frac32+(\frac72-p)-}=o(\mathcal{G}_{\lambda})\quad\mbox{as }|x|\to 0,\end{equation*}
which combined with $f_{\mathrm{reg}}\in L^{\infty}(\R^3)$ implies \eqref{limfzero}, thus proving (i).\par\smallskip

\emph{Step 3: the charge.} In view of \eqref{limfzero}, we have
\begin{equation*} \lim_{x\to 0}\mathcal{G}^{-1}_{\lambda}u(x)=\lim_{x\to 0}\mathcal{G}^{-1}_{\lambda}f(x)+q=q,\end{equation*}
namely \eqref{eq:qlim} holds. Moreover, observe that $u\in\mathcal{C}^2(\R^d)\Rightarrow q=0$. Conversely, if $q=0$, then $u$ satisfies
\begin{equation*}
(-\Delta+\lambda)u=\sigma|u|^{p-1}u\quad\mbox{on }\;\R^d,
\end{equation*}
and by elliptic regularity we get $u\in\mathcal{C}^2(\R^d)\cap H^2(\R^d)$, concluding the proof.
\end{proof}

We are now able to prove Theorems \ref{th:main_equivalence} and \ref{th:equivalence-weak}. To this aim, we are going to exploit the optimal Sobolev regularity of the regular component $f$, proved in \eqref{eq:improved_regularity} along the proof of Proposition \ref{prop:equivalence}. More specifically, in the trace regime we have $f\in H^{\frac{d}{2}+}(\R^d)$, which allows to define a boundary condition by selecting a specific value of $\alpha$, while in the free regime $f$ is unbounded at the origin, and thus no boundary condition can be determined. We also recall that, in the Trace Regime,  $u\in H_{\alpha}^{1}(\R^d)$ is a critical point of $S_{\lambda,\alpha}$ if and only if it satisfies \eqref{eq:point_elliptic1} as an identity in $H_{\alpha}^{-1}(\R^d)$; see e.g.~\cite[Proposition 4.2]{PoWa}.

\begin{proof}[Proof of Theorem \ref{th:main_equivalence}]
Let us start by proving the equivalence between conditions (i), (ii) and (iii). \par\smallskip
\emph{(i) $\Rightarrow$ (ii).} In view of Proposition \ref{prop:equivalence} and formula \eqref{eq:improved_regularity}, we can write $u=f+q\mathcal{G}_{\lambda}$, with $f\in H^s(\R^d)$, $s>\frac{d}{2}$, and $q\in\R$, and \eqref{eq:point_elliptic1} holds in $H^{s-2}(\R^d)$. By Sobolev embedding, $f\in\mathcal{C}(\R^d)$. Moreover, let us choose $\alpha=\infty$ when $q=0$, and the unique $\alpha\in\R$ determined by \eqref{eq:q-alpha} when $q\neq 0$. Owing to point (iii) of Proposition \ref{pr:hasp}, we have that $u\in H_{\alpha}^s(\R^d)$. Finally, since $H_{\alpha}^{s-2}(\R^d)\cong H^{s-2}(\R^d)$ for $s>\frac{d}{2}$, as it follows from point (i) of Proposition \ref{pr:hasp}, we deduce that \eqref{eq:point_elliptic1} holds in $H_{\alpha}^{s-2}(\R^d)$.
    \par\smallskip
    \emph{(ii) $\Rightarrow$ (iii)} Since $s>\frac d2\geqslant 1$, we deduce that $u\in H^1_\alpha(\R^d)$ and \eqref{eq:point_elliptic1}  holds in  $H^{-1}_\alpha(\R^d)$. Then $u$ is a critical point of $S_{\lambda,\alpha}.$\par\smallskip
    \emph{(iii) $\Rightarrow$ (i)} Since $u\in H_{\alpha}^1(\R^d)$ is critical point of $S_{\lambda,\alpha}$, it satisfies \eqref{eq:point_elliptic1} as an identity in $H^{-1}_\alpha(\R^d)$. Observe moreover that $H^{-1}_\alpha(\R^d)\subseteq H^{-1}(\R^d)$, as it follows by point (ii) of Proposition \ref{pr:hasp}. Then we get the thesis by applying the implication (ii) $\Rightarrow$ (i) of Proposition \ref{prop:equivalence}.\smallskip

To conclude, we observe that all the properties stated after the equivalence of (i), (ii) and (iii) have been proved along the present proof or in Proposition \ref{prop:equivalence}.
\end{proof}

\begin{proof}[Proof of Theorem \ref{th:equivalence-weak}] 
The equivalence (i) $\Leftrightarrow$ (ii) and the analysis of the case $q=0$ follow by Proposition \ref{prop:equivalence}. Let now $u=f+q\G_{\lambda}$, with $q\neq 0$. By \eqref{eq:improved_regularity}, $f\in H^s(\R^3)$ for any $s\in(\frac12,\frac{7}{2}-p)$. Since $\frac{7}{2}-p<\frac32$, point (ii) of Proposition \ref{pr:hasp} implies that $u\in H_{\alpha}^s(\R^3)$ for any such $s$ and for $\alpha\in\R$ arbitrary. We are left to show that $f$ is singular at the origin and satisfies the asymptotic \eqref{subleading}. Considering the first nonlinear correction $f_{\mathrm{rad}}^{(1)}$, defined in \eqref{def_reg_rad} and given explicitly by
\begin{equation}\label{eq:leading}
	f_{\mathrm{rad}}^{(1)}=
	\sigma|q|^{p-1}q\,\mathcal{G}_{\lambda}\ast\mathcal{G}_{\lambda}^p
	\sim
	\begin{cases}
		\ln{|x|},&p=2,\\
		|x|^{2-p},&p\in(2,3),
	\end{cases}
	\qquad\mbox{as }x\to 0,
\end{equation}
we equivalently need to prove that $f\sim f_{\mathrm{rad}}^{(1)}$ as $x\to 0$. To this aim, let us write $f=f_{\mathrm{rad}}+f_{\mathrm{reg}}$ as in \eqref{f_dec_rr}. Since $f_{\mathrm{reg}}$ is bounded, it is enough to show that \begin{equation}\label{first-correction}
	f_{\mathrm{rad}}\sim f_{\mathrm{rad}}^{(1)}\quad\mbox{as }x\to 0.\end{equation}
We recall that $f_{\mathrm{rad}}=f_{\mathrm{rad}}^{(K)}$, for a suitable $K\in\N^+$, is given explicitly by \eqref{def_reg_rad}.
Arguing as in \eqref{eq:stica}-\eqref{eq:slpit_fk}, Lipschitz estimates for the map $x\mapsto |x|^{p-1}x$ yield
$$\big|(-\Delta+\lambda)\big(f_{\mathrm{rad}}-f_{\mathrm{rad}}^{(1)}\big)\big|\lesssim \Big(|f_{\mathrm{rad}}^{(K-1)}+q\mathcal{G}_{\lambda}|^{p-1}+|q\mathcal{G}_{\lambda}|^{p-1}\Big)\big|f_{\mathrm{rad}}^{(K-1)}\big|.$$
Since $\mathcal{G}_{\lambda}\in L^{3-}(\R^3)$ and $f_{\mathrm{rad}}^{(K-1)}\in H^{(\frac72-p)-}(\R^3)\hookrightarrow L^{3+}(\R^3)$, in view of \eqref{lpgreen} and \eqref{eq:fH}, H\"older inequality then gives
$$(-\Delta+\lambda)\big(f_{\mathrm{rad}}-f_{\mathrm{rad}}^{(1)}\big)\in L^{\frac{3}{2p-3}-}(\R^3),$$
and using the Hardy-Littlewood-Sobolev inequality we eventually obtain
$$f_{\mathrm{rad}}-f_{\mathrm{rad}}^{(1)}\in \begin{cases}
L^{\infty}(\R^3)&\mbox{if }p\in[2,\frac52),\\
H^{\theta(p)-}(\R^3)&\mbox{if }p\in[\frac52,3),
\end{cases}$$
where $\theta(p)=\frac{13-4p}{2}$. In particular, if $p\in[2,\frac52)$, \eqref{first-correction} immediately follows. In the regime $p\in[\frac52,3)$, the weighted Sobolev embedding for radial functions (see e.g.~\cite[Proposition 1]{cho-ozawa}) yields
\begin{equation*}|f_{\mathrm{rad}}(x)-f_{\mathrm{rad}}^{(1)}(x)|\lesssim |x|^{-\frac32+\theta(p)-}=o(|x|^{2-p})\quad\mbox{as }|x|\to 0,\end{equation*}
which proves \eqref{first-correction} and concludes the proof.
\end{proof}

\begin{remark}\label{re:strong} 
Let us comment on the whole range of parameters $(s,\alpha)$ for which condition (ii) of Proposition \ref{prop:equivalence} holds.

In the regular case (i.e.~$q=0$), with the choice $\alpha=\infty$, every $s\in[1,2]$ when $d=2$, $s\in(\frac12,2]$ when $d=3$ is admissible. In addition, if $s<\frac d2$ or $u(0)=0$, then $\alpha$ can be chosen arbitrarily in $\overline{\R}$.

Let us consider now the singular case, namely $q\neq 0$. In the Free Regime, the possible choices are $s\in(\frac12,\frac72-p)$ and $\alpha\in\R$ arbitrary, as stated in Theorem \ref{th:equivalence-weak}. We now turn to the Trace Regime. Concerning the regularity level $s$, the admissible values are 
\begin{equation}
	s\in\begin{cases}\label{eq:admissible_s}
		[1,2]&d=2,\,p>1\\
		(\frac12,2]&d=3,\,p\in(1,\frac32)\\
		(\frac12,\frac72 - p)&d=3,\,p\in[\frac32,2).
	\end{cases}
\end{equation}
When $s>\frac{d}{2}$, the only choice for $\alpha$ is that given by \eqref{eq:q-alpha}, since we are in case (iii) of Proposition \ref{pr:hasp} on the characterization of adapted Sobolev spaces. The same is true when $d=3$, $s=\frac32$, since \eqref{tremezzi} and the inclusion $H^{3/2+}(\R^3)\subseteq \widetilde{H}^{3/2}(\R^3)$ imply that the regular component $f$ and the charge $q$ must be related through condition \eqref{eq:q-alpha}. When $s=1$, $d=2$ or $s\in(\frac12,\frac32)$, $d=3$, condition (ii) holds instead for arbitrary $\alpha\in\R$, since we are in case (ii) of Proposition \ref{pr:hasp}.

Finally, we point out that for every $s\in(\frac12,\frac32]$ we have 
$$\mbox{\eqref{eq:point_elliptic1} holds in } H_{\alpha}^{s-2}(\R^3)\Leftrightarrow \alpha\mbox{ is given by \eqref{eq:q-alpha}}.$$
The implication ($\Leftarrow$) is immediate, since by Theorem \ref{th:main_equivalence} equation \eqref{eq:point_elliptic1} actually holds in $H_{\alpha}^{-\frac{1}{2}+}(\R^3)$ if $\alpha$ is given by \eqref{eq:q-alpha}. To check the implication ($\Rightarrow$), we use the following argument. We have that $\operatorname{span}\{\mathcal{G}_{\lambda}\}$ is a closed subspace of $H_{\alpha}^{2-s}(\R^3)$, as follows from Proposition \ref{pr:hasp} and point (i) of Remark \ref{re:transition}. Hence, if \eqref{eq:point_elliptic1} holds in  $H_{\alpha}^{s-2}(\R^3)$, then in particular  
\begin{equation}\label{eq:seteg}
\big\langle (-\Delta_{\alpha}+\lambda)u-\sigma|u|^{p-1}u,\mathcal{G}_{\lambda}\big\rangle_{H_{\alpha}^{s-2}\times H_{\alpha}^{2-s}}=0.
\end{equation}
Now we recall \eqref{eq:test_green}, namely
	\begin{equation*}
		\big\langle (-\Delta_{\alpha}+\lambda)u,\mathcal{G}_{\lambda}\big\rangle_{H_{\alpha}^{s-2}\times H_{\alpha}^{s-2} }=\beta_{\alpha}(\lambda)q.
		\end{equation*}
	On the other hand, 
	\begin{equation}\label{tr}
		\begin{split}
	\big\langle \sigma|u|^{p-1}u,\mathcal{G}_{\lambda}\big\rangle_{H_{\alpha}^{s-2}\times H_{\alpha}^{2-s}}&=\int_{\R^d} \sigma|u|^{p-1}u\mathcal{G}_{\lambda}\,\ud x\\ &=\int_{\R^d} (-\Delta+\lambda)f\,\cdot\,\mathcal{G}_{\lambda}\,\ud x=f(0),
	\end{split}
	\end{equation}
	where we used that $f\in\mathcal{C}_{b}(\R^3)$ and $\mathcal{G}_{\lambda}^{p+1}\in L^1(\R^3)$ in the first step, identity \eqref{eq:ellif} in the second step, and $f\in H^{\frac{d}{2}+}(\R^3)$ in the last step. Using \eqref{eq:test_green}, \eqref{tr}, we deduce that \eqref{eq:seteg} is satisfied if and only if $\beta_{\alpha}(\lambda)q=f(0)$, as desired.
\end{remark}

\subsection{Existence of critical points of the action}\label{sec:var}
We prove here Theorem \ref{th:main_critical_points} on the existence of infinitely many singular, radial critical points of the action functional defined by \eqref{eq:action}.

From now on, to ease the notation, we set $\D:=H_{\alpha}^1(\R^d)$. To start with, we recall the explicit expression \eqref{eq:quad_form1} for the quadratic form $\mathcal{Q}_{\alpha}$ associated to $-\Delta_{\alpha}$, $\alpha\in\R$, with the convenient choice $\omega=\lambda$: given $u=f+q\mathcal{G}_{\lambda}\in\D$, we have
\begin{equation}\label{eq:quad_form}
	\mathcal{Q}_{\alpha}(u)=\int_{\R^d}\left(|\nabla f|^2+\lambda|f|^2-\lambda|u|^2\right)\,dx+\beta_{\alpha}(\lambda)q^2.
\end{equation}
Then we can write
\begin{equation*}
	S_{\lambda,\alpha}(u)=\frac{1}{2}\int_{\R^d}\left(|\nabla f|^2+\lambda|f|^2\right)\,dx+\frac{\beta_\alpha(\lambda)}{2}q^2-\frac{1}{p+1}\int_{\R^d}|u|^{p+1}\,dx.
\end{equation*}
Moreover, for $\lambda>\lambda_\alpha$, the quantity  
 \begin{equation*}
\|u\|_{\D}:=\left(\int_{\R^d}\left(|\nabla f|^2+\lambda|f|^2\right)\,dx+\beta_\alpha(\lambda)q^2\right)^{\frac{1}{2}}
 \end{equation*}
induces a norm on $\D$ equivalent to the one defined in \eqref{eq:singular_spaces}.

We focus in our analysis to radial solutions: on the one hand, this provides sufficient compactness to deduce the Palais-Smale property (see Lemma \ref{lem:pal-smale} below), on the other hand, radial (non-zero) critical points turn out to be singular, as shown by the following.

\begin{lemma}\label{le:sing_crit}
Fix $\alpha\in\R$, $\lambda>0$. If $u=f+q\mathcal{G}_{\lambda}\in\D$ is a non-zero, radial critical point of $S_{\alpha,\lambda}$, then $u$ is singular, namely $q\neq 0$.
\end{lemma}

\begin{proof}
Let $u\not\equiv 0$ be a radial critical point of $S_{\lambda,\alpha}$, and suppose by contradiction that $q=0$. Then by Theorem \ref{th:main_equivalence} $u\in\mathcal{C}^2(\R^d)$ solves the system
	\begin{equation*}
		\begin{cases}
			(-\Delta+\lambda)u=|u|^{p-1}u,\\
			u(0)=0.
		\end{cases}
	\end{equation*}
	Since $u$ is radial, we can write $u(x)=\tilde{u}(|x|)$ for some function $\tilde{u}:[0,+\infty)\to \R$ satisfying the system
	\begin{equation*}
		\begin{cases}
			-\tilde{u}''(r)-\frac{d-1}{r}\tilde{u}'(r)+\lambda\tilde{u}(r)=|\tilde{u}(r)|^{p-1}\tilde{u}(r),\\
			\tilde{u}(0)=0.
		\end{cases}
	\end{equation*}
Since $u\in \mathcal{C}^2(\R^d)$, then also $\tilde{u}'(0)=0$. However, the Cauchy problem 
\begin{equation}\label{sing_cauchy}
	\begin{cases}
		-\tilde{u}''(r)-\frac{d-1}{r}\tilde{u}'(r)+\lambda\tilde{u}(r)=|\tilde{u}(r)|^{p-1}\tilde{u}(r),\quad r>0\\
		\tilde{u}(0)=\tilde{u}'(0)=0\\
	\end{cases}
\end{equation}
admits the only solution $\tilde{u}\equiv 0$, yielding a contradiction. 
In order to prove that $\tilde{u}\equiv 0$ is the only solution to \eqref{sing_cauchy}, let us consider the Lyapunov function
\begin{equation*}
	E(r):=\frac{\tilde{u}'(r)^2}{2}+\frac{|\tilde{u}(r)|^{p+1}}{p+1}-\frac{\lambda \tilde{u}(r)^2}{2},\quad r>0.
\end{equation*}
A straightforward computation shows that $E$ is non-increasing on $\R^+$, so that $E(r)\leqslant E(0)=0$ and in particular 
\begin{equation}
	\label{eq:abs-u'}
	|\tilde{u}'(r)|\leqslant\sqrt{\lambda}|\tilde{u}(r)|\quad \text{for every} \quad r>0.
\end{equation} 
If we introduce the function $\omega(r):= e^{-\sqrt{\lambda}r}|\tilde{u}(r)|$, then by using \eqref{eq:abs-u'} one can check that $\omega$ is non-increasing on $\R^+$. Since $\lim_{r\to 0^+} \omega(r)=0$ and $\omega(r)\geqslant 0$ for every $r>0$, it follows that $\omega\equiv 0$ on $\R^+$. This forces $\tilde{u}\equiv 0$ on $\R^+$, concluding the proof.
\end{proof}

In what follows, we thus look for the existence of critical points of $S_{\lambda,\alpha}$ on the Banach space
\begin{equation*}
	\D_{\mathrm{rad}}:=\{u\in \D\,:\,u\,\, \text{is radially symmetric}\},
\end{equation*}
endowed with the $\|\cdot\|_{\D}$ norm. 

We strongly rely on the method developed by Ambrosetti and Rabinowitz in \cite{AR-73}. Let us recall their framework. Let $E$ be an infinite dimensional Banach space over $\R$ and $J:E\to \R$ be a functional of class $\mathcal{C}^1$ satisfying the following hypotheses:
\begin{itemize}
	\item[(i)] $J(0)=0$; 
     \item[(ii)] there exists $\delta,\rho>0$ such that $J>0$ in $B_\rho\setminus\{0\}$ and $J\geqslant \delta$ on $\partial B_\rho$;
	\item[(iii)] for every $E'\subset E$,  $\operatorname{dim}\,(E')<+\infty$, the set $E'\cap\{J\geqslant 0\}$ is bounded;
	
    \item[(iv)] $J(u)=J(-u)$ for every $u\in E$;
	\item[(v)] If $(u_n)_n\subset E$ satisfies 
	\begin{equation*}
		0<J(u_n)\leqslant C,\quad \text{and}\quad J'(u_n)\to 0 \quad\text{as}\quad n\to+\infty,
	\end{equation*}
	then $(u_n)_n$ admits a convergent subsequence.
\end{itemize}

We say that the functional $J$ has the \emph{strong} mountain pass geometry if it satisfies (i)-(iii): indeed, the usual mountain pass geometry property corresponds to consider in (iii) only the case when $\operatorname{dim}(E')=1$. On the other hand,  hypothesis (v) is customary called the Palais-Smale condition.

The following result, proved in \cite{AR-73}, concerns the existence of infinite critical points of $J$ under the above assumptions.

\begin{proposition}
	\label{prop:inf-crit-j}
	If (i)-(v) hold, then $J$ admits infinitely many critical points.
\end{proposition}

In the next lemma, we prove that the functional $S_{\lambda,\alpha}:\D_{\text{rad}}\to \R$
has the strong mountain pass geometry.

\begin{lemma}
	\label{lem:mount-geom}
	Let $\alpha\in\R$ and $\lambda>\lambda_\alpha$. Then the functional $S_{\lambda,\alpha}:\D_{\text{rad}}\to\R$ has the strong mountain pass geometry.
\end{lemma}
\begin{proof}
	Condition (i) is obviously satisfied, hence we start by proving (ii). Let  $u=f+q\G_\lambda\in B^*_\rho$, i.e. 
	\begin{equation*}
		\|u\|_{\D}^2=\|\nabla f\|_2^2+\lambda\|f\|_2^2+\beta_\alpha(\lambda)q^2\leqslant \rho^2.
	\end{equation*}
	We observe that
	\begin{equation}
		\label{eq:u2-uD}
		\|u\|_{2}^2\leqslant 2\left(\|f\|_2^2+|q|^2\|\G_\lambda\|_2^2\right)\leqslant \frac{2}{\lambda}\|u\|_{\D}^2+\frac{2\|  \G_\lambda\|_2^2}{\beta_\alpha(\lambda)}\|u\|_{\D}^2\lesssim \|u\|_{\D}^2.
	\end{equation}
Therefore, by applying the Gagliardo-Nirenberg inequality in \cite[Equation (2.11)]{CFN-21} with $s=d\left(\frac{1}{2}-\frac{1}{p+1}\right)$
and using \eqref{eq:u2-uD}, we get
	\begin{equation}
	\label{eq:(ii) a-r}	
	\begin{split}
	S_{\lambda,\alpha}(u)&\geqslant \frac{1}{2}\|u\|_{\D}^2-\frac{C}{p+1}\|u\|_{\D}^{\frac{d}{2}(p-1)}\|u\|_2^{p+1-\frac{d}{2}(p-1)}\\
	&\geqslant \frac{1}{2}\|u\|_{\D}^2-\frac{C}{p+1}\|u\|_{\D}^{p+1}>0
	\end{split}
	\end{equation}
	for $\rho$ sufficiently small, entailing (ii).
	
	We focus now on (iii). Let $\D_{\text{fin}}$ be a finite dimensional subset of $\D_{\text{rad}}$, and let $\{u_1,\dots,u_n\}$ be an orthonormal basis of $\D_{\text{fin}}$ with respect to the inner product induced by $\|\cdot\|_\D$. This entails in particular that every element $g\in \D_{\text{fin}}$ can be written as $g=\sum_{i=1}^n t_i u_i$, with $t_i\in \R$ for every $i=1,\dots,n$, and $\|g\|_{\D}^2=\sum_{i=1}^n t_i^2$. Let us recall  that all the norms on the finite dimensional space $\D_{\text{fin}}$ are equivalent, in particular the two norms $\|\cdot\|_\D$ and $\|\cdot\|_{L^2(\R^d)}$: this implies that there exists $C_0>0$ such that  
    \begin{equation}
    \label{eq:norm-equiv}
\|g\|_{L^2(\R^d)}^2\geqslant C_0\|g\|_{\D}^2 \quad \forall\, g\in \D_{\text{fin}}.   
    \end{equation} 
    Consider now $g\in \D_{\text{fin}}$ with $\|g\|_\D=1$ and observe that 
    \begin{equation}
    \label{eq:lowboundp}
\int_{\R^d}|Rg|^{p+1}\,dx>KR^2\int_{\{|g|>K^{\frac{1}{p-1}}R^{-1}\}}|g|^2\,dx,
    \end{equation}
    with $K:=\frac{2(p+1)}{C_0}$. In particular, there exists $R^\star>0$ such that for every $R\geqslant R^\star$
    \begin{equation}
    \label{eq:bound2}
        \int_{\{|g|>K^{\frac{1}{p-1}}R^{-1}\}}|g|^2\,dx\geqslant \frac{1}{2}\|g\|_{L^2(\R^d)}^2\geqslant \frac{C_0}{2},
    \end{equation}
    where we have used \eqref{eq:norm-equiv} and the fact that $\|g\|_\D=1$.
    By combining
 \eqref{eq:lowboundp} and \eqref{eq:bound2},  there results that for every $R\geqslant R^\star$
    \begin{equation}
        \frac{1}{p+1}\|Rg\|_{L^{p+1}(\R^d)}^{p+1}\geqslant R^2.
    \end{equation}
    Therefore, we have that
    \begin{equation}
		\label{eq:s-om-lau}
S_{\lambda,\alpha}(g)\leqslant\frac{1}{2}R^2-R^2=-\frac{R^2}{2}<0\quad \text{ for $R\geqslant R^\star$}
	\end{equation} 
 	entailing that $\D_{\text{fin}}\cap\{S_{\lambda,\alpha}\geqslant 0\}$ is bounded and concluding the proof of (iii).
\end{proof}

In the next lemma, we prove that $S_{\lambda,\alpha}$ satisfies the Palais-Smale condition.

\begin{lemma}
	\label{lem:pal-smale}
	Let $\alpha\in\R$, $\lambda>\lambda_{\alpha}$. Then $S_{\lambda,\alpha}$ satisfies the Palais-Smale condition.
\end{lemma}
\begin{proof}
	Let $(u_n)_n$ be a sequence satisfying $0<S_{\lambda,\alpha}(u_n)\leqslant C$ and $S'_{\lambda,\alpha}(u_n)\to 0$ as $n\to+\infty$. Then
	\begin{equation}
		\label{eq:s'un-un}
		\langle S'_{\lambda,\alpha}(u_n), u_n\rangle_{\D'\times\D}=\|u_n\|_{\D}^2-\|u_n\|_{p+1}^{p+1},
	\end{equation}
	from which
	\begin{equation*}
		S_{\lambda,\alpha}(u_n)=\frac{1}{2}\|u_n\|_{\D}^2-\frac{1}{p+1}\|u_n\|_{p+1}^{p+1}=\frac{1}{p+1}\langle S'_{\lambda,\alpha}(u_n), u_n\rangle_{\D'\times\D}+\frac{p-1}{2p+2}\|u_n\|_{\D}^{2}.
	\end{equation*}
	In particular, since $\|S'_{\lambda,\alpha}(u_n)\|_{\D'}\to 0$ it follows that
	\begin{equation}
		\begin{split}
			\frac{p-1}{2p+2}\|u_n\|_{\D}^2&=S_{\lambda,\alpha}(u_n)-\frac{1}{p+1}\langle S'_{\lambda,\alpha}(u_n), u_n\rangle_{\D'\times\D}\\
			&\leqslant \left|S_{\lambda,\alpha}(u_n)\right|+\frac{1}{p+1}\left|S'_{\lambda,\alpha}(u_n), u_n\rangle_{\D'\times\D}\right|\\
			&\leqslant \left|S_{\lambda,\alpha}(u_n)\right|+\frac{1}{p+1}\|S'_{\lambda,\alpha}(u_n)\|_{\D'}\|u_n\|_{\D}\\
			&\leqslant C+K\|u_n\|_{\D},
		\end{split}
	\end{equation}
	thus the sequence $(u_n)_n$ is bounded in $\D$. This implies  that, given the representation $u_n=f_n+q_n\G_\lambda$, then  $f_n$ and $q_n$ are bounded in $H^1(\R^d)$ and in $\R$ respectively. Therefore, there exists $f\in H^1(\R^d)$ and $q\in \R$ such that, up to subsequences, $f_n\deb f$ weakly in $H^1(\R^d)$ and $q_n\to q$ in $\R$. Since $f_n$ are radially symmetric, then by \cite[Proposition 1.7.1]{cazenave} $f_n\to f$ strongly in $L^{r}(\R^d)$ for every $r\in
	(2,+\infty)$ if $d=2$ and $r\in(2,3)$ if $d=3$, in particular $f_n\to f$ strongly in $L^{p+1}(\R^d)$. Writing $u:=f+q\G_\lambda$, it follows that $u_n\to u$ strongly in $L^{r}(\R^d)$ for every $r\in(2,+\infty)$ if $d=2$ and $r\in(2,3)$ if $d=3$.
	
	On the one hand, by \eqref{eq:s'un-un}
	\begin{equation}\label{eq:normu-leq-up}
		\begin{split}
\|u\|_\D^2\leqslant \liminf_n \|u_n\|_{\D}^2\!=\!\liminf_n\left\{\|u_n\|_{p+1}^{p+1}+\langle S'_{\lambda,\alpha}(u_n), u_n\rangle_{\D'\times\D} \right\}\!=\!\|u\|_{p+1}^{p+1}.
\end{split}
\end{equation}
	On the other hand, for every $v=\varphi+Q\G_{\lambda}\in \D$,
	\begin{equation}
		\begin{split}
			\langle S'_{\lambda,\alpha}(u)&-S'_{\lambda,\alpha}(u_n), v\rangle_{\D'\times\D}=\int_{\R^d}\left[\left(\nabla f-\nabla f_n\right)\cdot \nabla \varphi+\lambda\left(f-f_n\right)\cdot \varphi\right]\,dx\\
			&+\beta_{\alpha}(\lambda)(q-q_n)Q-\int_{\R^d}\left(|u|^{p-1}u-|u_n|^{p-1}u_n\right)v\,dx.
		\end{split}
	\end{equation}
	The first two terms converge to $0$ since $f_n\deb f$ weakly in $H^1(\R^d)$ and $q_n\to q$ in $\R$. The last term can be estimated for example as follows:
	\begin{equation}
		\begin{split}
\left|\int_{\R^d}\left(|u|^{p-1}u-|u_n|^{p-1}u_n\right)v\,dx\right|&\leqslant p\int_{\R^d}\sup\left\{|u|, |u_n|\right\}^{p-1}|u-u_n||v|\,dx\\
			&\leqslant 2p\|v\|_{q_1}\|u-u_n\|_{q_2}\|u_n\|_{(p-1)q_3}^{p-1},
		\end{split}
	\end{equation}
	where the $q_i$ satisfy $\sum_i \frac{1}{q_i}=1$. It is straightforward that the right-hand side vanishes when $d=2$, independently on the choice of $q_i>2$. If instead $d=3$, then there are restrictions on the integrability of functions belonging to $\D$, in particular $2<q_1,q_2<3$ and $\frac{2}{p-1}<q_3<\frac{3}{p-1}$, that can be rewritten as $\frac{1}{3}<\frac{1}{q_1},\frac{1}{q_2}<\frac{1}{2}$ and $\frac{p-1}{3}<\frac{1}{q_3}<\frac{p-1}{2}$: an admissible choice is given by $\frac{1}{q_1}=\frac{1}{q_2}:=\frac{1}{3}+\frac{\eps(p)}{2}$ and $\frac{1}{q_3}:=\frac{1}{3}-\eps(p)$, with $\eps(p):=\frac{1}{12}p^2-\frac{7}{12}p+\frac{5}{6}$.
	
	Therefore, since $\langle S'_{\lambda,\alpha}(u)-S'_{\lambda,\alpha}(u_n), v\rangle_{\D'\times\D}\to 0$ for every $v\in \D$, it holds that $S_{\lambda,\alpha}'(u)=0$ and $\|u\|_{\D}^2=\|u\|_{p+1}^{p+1}$. This, together with \eqref{eq:normu-leq-up}, entails that $u_n\to u$ strongly in $\D$, concluding the proof.  
\end{proof}

We can now prove the existence of infinitely many singular, critical points.

\begin{proof}[Proof of Theorem \ref{th:main_critical_points}]
 By virtue of Lemmas \ref{lem:mount-geom}-\ref{lem:pal-smale} and the fact that property (iv) holds (i.e.~$S_{\lambda,\alpha}$ is an even functional), one can apply Proposition  \ref{prop:inf-crit-j} to get the existence of infinitely many radial critical points to \eqref{eq:action}, which are also singular in view of Lemma \ref{le:sing_crit}.
\end{proof}

\subsection{Structure of positive solutions and existence of nodal solutions}\label{sec:strunod}
We prove here Theorems \ref{th:structure} and \ref{th:nodal}, respectively on the structure of  positive solutions and the existence of nodal solutions to \eqref{eq:singular_elliptic} in dimension $d=2$. We preliminary collect some useful material. 

Let us start with the following rigidity result, which guarantees that positive solutions to \eqref{eq:singular_elliptic} are radially symmetric.

 \begin{proposition}\label{pr:radial}
	Fix $d\geqslant 2$, $\sigma=\pm 1$, $\lambda>0$, $p>1$, and let $u$ be a positive solution to \eqref{eq:singular_elliptic} with $|u(x)|\to 0$ as $|x|\to\infty$. Then $u$ is radially symmetric around the origin (up to a suitable translation if $u$ is regular).
\end{proposition}
In the case of regular solutions, the above result was obtained (for a large class of local nonlinearities) in the seminal paper by Gidas-Ni-Nirenberg \cite{GNN}, using the technique of ``moving planes''. The extension to the singular case can be found in \cite[Theorem 2.4]{Veron-book}. Strictly speaking, the result in \cite{Veron-book} deals with elliptic equations on a ball with Dirichlet boundary conditions. Still, the first step ``near infinity'' of the moving plane procedure is as in \cite{GNN}, then the proof proceeds as in \cite{Veron-book}.

Next, we present a crucial uniqueness result for equation \eqref{eq:point_elliptic1} in the source case, obtained in \cite{fukaya-uniqueness}, which also proves the non-degeneracy of the action ground state.

 \begin{proposition}\label{pr:uniqueness}
Fix $d=2$, $\sigma=1$, $\alpha\in\overline{\R}$, $\lambda>\lambda_{\alpha}$ and $p>1$. There exists exactly one positive, radial solution $u$ to \eqref{eq:point_elliptic1}. Moreover,
\begin{itemize}
\item[(i)] if $\alpha\in\R$, then $\mathsf{G}_{\lambda,\alpha}=\{\pm u\}$;
\item[(ii)] $\mathsf{G}_{\lambda,\infty}=\{\pm u(\cdot-x_0)\,|\,x_0\in\R^2\}$.
\end{itemize}
\end{proposition}
In the regular case ($\alpha=\infty$) the above result is classical \cite{Kwong}, and holds in any dimension (for suitable ranges of $p$). As already mentioned, the proof in \cite{fukaya-uniqueness} for the singular regime ($\alpha\in\R$) instead strongly relies on the fact that $d=2$.

Finally, we state a simple lemma showing that any radial, non-zero solution to \eqref{eq:singular_elliptic} that vanishes at some point away from the origin must be nodal.

\begin{lemma}\label{le:zin}
	Fix $d\geqslant 2$, $\sigma=\pm 1$, $\lambda>0$, $p>1$. Let $u\not\equiv 0$ be a radial solution to \eqref{eq:singular_elliptic}, with $u(x_0)=0$ for some $x_0\neq 0$. Then $u$ is nodal, namely $u^+,\,u^-\not\equiv 0$.
\end{lemma}
\begin{proof}
	Let us write $u(x)=\tilde{u}(|x|)$ for some non-zero function $\tilde{u}:(0,+\infty)\to\R$. Set also $r_0:=|x_0|$, so that $\tilde{u}(r_0)=0$. If $\tilde{u}'(r_0)=0$, then $u$ would satisfy
	\begin{equation*}
	\begin{cases}
		-\tilde{u}''(r)-\frac{d-1}{r}\tilde{u}'(r)+\lambda\tilde{u}(r)=\sigma|\tilde{u}(r)|^{p-1}\tilde{u}(r),\quad r>0\\
		\tilde{u}(r_0)=\tilde{u}'(r_0)=0,
	\end{cases}
	\end{equation*}
	and by uniqueness $\tilde{u}(r)=0$ for all $r>0$, a contradiction. Thus $\tilde{u}'(r_0)\neq 0$, and as a consequences $\tilde{u}$ changes sign in a neighborhood of $r_0$.
\end{proof}

We are now able to prove Theorems \ref{th:structure} and \ref{th:nodal}.
\begin{proof}[Proof of Theorem \ref{th:structure}]
In view of Proposition \ref{pr:radial}, $u$ must be radial. Moreover, by Theorem \ref{th:main_equivalence}, $u\in H_{\alpha}^2(\R^2)$ for a suitable choice of $\alpha\in\overline{\R}$, and it satisfies \eqref{eq:point_elliptic1} as an identity in $L^2(\R^2)$. Let us show that $\lambda>\lambda_{\alpha}$. Since $\lambda_{\infty}=0$, we only need to consider the case $\alpha\in\R$. Suppose that $\lambda\leqslant \lambda_{\alpha}$: testing \eqref{eq:point_elliptic1} with $\mathcal{G}_{\lambda_{\alpha}}$, and recalling that $(-\Delta_{\alpha}+\lambda_{\alpha})\mathcal{G}_{\lambda_{\alpha}}=0$, we get
\begin{equation*}
0\geqslant(\lambda-\lambda_{\alpha})\big\langle u,\mathcal{G}_{\lambda_{\alpha}}\big\rangle=\big\langle(-\Delta_{\alpha}+\lambda)u,\mathcal{G}_{\lambda_{\alpha}}\big\rangle=\big\langle(|u|^{p-1}u,\mathcal{G}_{\lambda_{\alpha}}\big\rangle>0,
\end{equation*}
which is a contradiction. Since $\lambda>\lambda_{\alpha}$, the thesis follows from Proposition \ref{pr:uniqueness} and the aforementioned fact, proved in \cite{ABCT-2d}, that ground states of $S_{\lambda,\alpha}$, $\alpha\in\R$, are singular.
\end{proof}

\begin{remark}\label{rem:crucial}
In the above proof, it was crucial that in the source case (as opposed to the absorption one) no a priori upper bound on the solution $u$ is required in the equivalence result provided by Theorem \ref{th:main_equivalence}.
\end{remark}

\begin{proof}[Proof of Theorem \ref{th:nodal}]
In view of \eqref{eq:beta_alfa} and the relation $\beta_{\alpha}(\lambda_{\alpha})=0$, we deduce that $\inf_{\alpha\in\R}\lambda_{\alpha}=0$. In particular, we can choose $\alpha\in\R$ such that $\lambda>\lambda_{\alpha}$. Theorem \ref{th:main_critical_points} then provides the existence of infinitely many radial, singular critical points of $S_{\lambda,\alpha}$, which are solutions to \eqref{eq:point_elliptic1} and  hence also to \eqref{eq:singular_elliptic} in view of Theorem \ref{th:main_equivalence}. Since by Proposition \ref{pr:uniqueness} there exist a unique positive, radial solution $u_{\alpha}$ to \eqref{eq:point_elliptic1}, all these critical points but $\pm u_{\alpha}$ vanish at some point, whence are nodal in view of Lemma \ref{le:zin}.
\end{proof}

\section{Further extensions and remarks}\label{sec:finalremarks}

\subsection{Other nonlinearities and generalizations}
In our equivalence result, the pure-power nonlinearity $|u|^{p-1}u$ can be replaced, with no substantial changes to the proof, by more general nonlinearities (possibly non-monotone and of indefinite sign). This leads to elliptic problems that still admit singular solutions, which can be characterized in terms of Schr\"odinger operators with $\delta$-type (point) interactions. In this direction, we mention the recent paper \cite{PoWa}, where the existence of a non-trivial solution to $-\Delta_{\alpha}u=g(u)$, for fairly general nonlinearities, is proved via variational methods.

Once the correspondence between \eqref{eq:singular_elliptic} and \eqref{eq:point_elliptic1} is clarified in the model case of the Euclidean Laplacian with a single point interaction, several extensions to more general settings become natural. Although we do not pursue them here, we mention the following open directions:
\begin{itemize}
 \item[-] working on bounded domains or on Riemannian manifolds;
 \item[-] considering elliptic operators with variable coefficients or pseudodifferential operators (e.g.\ the fractional Laplacian);
 \item[-] allowing singularities supported on multiple points (rather than a single one), or on lower-dimensional sets such as curves or surfaces.
\end{itemize}

\subsection{The absorption case}\label{ssa}
In the absorption case, when $d=3$ and $p>1$, it is known that there are no regular ground states, i.e.~positive regular solutions to \eqref{eq:singular_elliptic} vanishing at infinity (see \cite{NS1_86}). On the other hand, when $d=3$ and $1<p<2$, there exist positive solutions to \eqref{eq:point_elliptic1} (see \cite{CaCle94}), and then also to \eqref{eq:singular_elliptic} by our equivalence result. More generally, in the regime covered by Theorem~\ref{th:main_equivalence}, singular solutions to \eqref{eq:singular_elliptic} with $\sigma=-1$ do always exist; see, e.g., \cite{Veron-book,CirsteaDu07}. From the operator-theoretic perspective emphasized in the present work, this phenomenon is ultimately related to the simultaneous presence of a defocusing nonlinearity and, at the linear level, an attractive point interaction. Recall indeed that the operator $-\Delta_\alpha$ has a negative eigenvalue for $\alpha<0$. We expect a similar mechanism to hold also in dimension $d=2$, where a negative eigenvalue of $-\Delta_\alpha$ exists for every $\alpha\in\R$. For information about positive and nodal singular solutions to \eqref{eq:singular_elliptic} with $\sigma=-1$ in $\mathbb R^2$, see \cite{ChenMatanoVeron89, CirsteaDu07}.

\subsection{Complex solutions}\label{re:complesso}
\sloppy
In quantum mechanics, Schr\"odinger operators with point interactions are naturally realized on $L^2(\R^d;\C)$. Accordingly, one may also consider complex-valued solutions to equation \eqref{eq:singular_elliptic}, and the equivalence results provided by Theorems~\ref{th:main_equivalence}-\ref{th:equivalence-weak} still hold. Indeed, in view of the gauge invariance of the nonlinear term, it suffices to impose condition \eqref{eq:below-green} on both the real and the imaginary parts of $u$. Interest in complex-valued solutions already arises in the regular setting, in particular in the study of vortices for nonlinear Schr\"odinger equations and related models, such as the Gross--Pitaevskii and Ginzburg--Landau equations; see, e.g., \cite{CEQ},\cite[Chapter 15]{fibich} and references therin. For instance, when $d=2$, \eqref{eq:singular_elliptic} admits regular solutions of the form $f(r)e^{\ii n\theta}$, where $(r,\theta)$ are polar coordinates and $n\in\Z$ is the vortex degree. The study of non-constant phase solutions, including vortex-type profiles, in the singular setting is an interesting direction for future research. The complex-valued version of Theorem~\ref{th:main_equivalence} provides an operator-theoretic framework for such an analysis.

\subsection{Time-dependent problems}\label{re:timedependent} 
Singular solutions to semilinear heat equations have been widely studied, see e.g.~\cite{Oswald88, Hirata14} and the references therein. On the other hand, semilinear heat equations with point interactions have attracted attention in the literature only recently, see Section \ref{re:heat} below for an explicit example. Beyond the heat flow, other dynamics can be considered. For the semilinear Schr\"odinger equation with point interactions, we refer to \cite{MOS-hartree, CFN-21, FGI, CFN-23, GeoRas-2D}. The analysis of nonlinear wave flow with ``zero range'' interactions is instead, to the best of our knowledge, completely open.

\subsection{Other approaches}\label{re:heat}
The operator-theoretic characterization of the elliptic equation \eqref{eq:singular_elliptic} provided by Theorems~\ref{th:main_equivalence}-\ref{th:equivalence-weak} makes available a wide range of tools. For example, the self-adjointness of $-\Delta_\alpha$ allows one to define the heat semigroup $\{e^{t\Delta_\alpha}\}_{t>0}$ and, in turn, by means of parabolic estimates (see, e.g., \cite{Bargera}), the flow of the singular semilinear heat equation
\begin{equation}\label{eq:nlh}
v_t-\Delta_\alpha v+\lambda v = \sigma|v|^{p-1}v.
\end{equation}
Stationary solutions of this flow are precisely solutions of \eqref{eq:point_elliptic1}. Under additional compactness and dissipativity assumptions, one may recover stationary solutions as $\omega$-limit points of parabolic trajectories. In the regular setting, the above connection has been used in \cite{CMT-2000}  to construct multiple nodal solutions.

\subsection{Higher dimensions}\label{re:dprime}
As already noticed, in dimensions $d\geqslant4$, analogues of Theorems~\ref{th:main_equivalence}-\ref{th:equivalence-weak} cannot hold within the same $L^2$ self-adjoint extension framework, since the operator
\begin{equation*}
S:=-\Delta\big|_{\mathcal{C}_0^{\infty}(\R^d\setminus\{0\})}
\end{equation*}
is essentially self-adjoint on $L^2(\R^d)$ for $d\geqslant4$, meaning that nontrivial $L^2$-based point interactions exist only in dimensions $d=1,2,3$.
A related correspondence in higher dimension may still be possible, but it would have to be formulated in a larger setting (using e.g.~weighted Sobolev spaces), where Green-type singularities are admitted as states. See also \cite{DerezinskiPoint} for a different point of view on the treatment of point interactions in higher dimension.

\appendix

\section{Proof of the Brezis-Lions Lemma}\label{app}
We prove here Proposition \ref{lem:Brezis-Lions}, following the same lines as in the seminal paper \cite{Brezis-Lions-Lemma} by Brezis-Lions.\par\smallskip
\emph{Step 1:~$u\in L^{1}_{\mathrm{loc}}(\R^{d})$.}
Denote by $\overbar{u}(r)$ the spherical mean of $u$ on $\partial B_{r}$, namely
\begin{equation*}
\overbar{u}(r):=\frac{1}{|\partial B_{r}|}\int_{\partial B_{r}} u(x)\,d\sigma(x),
\end{equation*}
where $d\sigma$ is the $(d-1)$-dimensional Hausdorff measure. Since $u\in L^{1}_{\mathrm{loc}}(\R^{d}\setminus\{0\})$ and hypotheses $(i)$ and $(ii)$ hold, there results that $\overbar{u}\in \mathcal{C}^{1}((0,+\infty))$, 
\begin{equation}
 \frac{d}{dr}\left(r^{d-1}\frac{d \overbar{u}}{dr}\right)\in L^{1}_{\text{loc}}((0,+\infty)),
\end{equation}
\begin{equation}
\label{eq:elliptic-mean}
-\frac{1}{r^{d-1}}\frac{d}{dr}\left(r^{d-1}\frac{d \overbar{u}}{dr}\right)+\lambda \overbar{u}\geqslant\overline{g}\quad \text{on}\quad (0,+\infty).
\end{equation}
Let now $0<L\leqslant R$, with $R$ as in hypothesis $(iii)$. Multiplying \eqref{eq:elliptic-mean} by $r^{d-1}$ and integrating over the interval $(r,L)$, we get 
\begin{equation}
\label{eq:est-du/dr}
\begin{split}
-r^{d-1}\frac{d\overbar{u}}{dr}(r)&\leqslant\lambda \int_{r}^{L}s^{d-1}\overbar{u}(s)\,ds-\int_{r}^{L}s^{d-1}\overline{g}(s)\,ds-L^{d-1}\frac{d\overbar{u}}{dr}(L)\\
&\leqslant \lambda \int_{r}^{L}s^{d-1}\overbar{u}(s)\,ds + C,
\end{split}
\end{equation}
where the constant $C$ is, from now on, independent of $r\in(0,L)$. Next, let us set
\begin{equation*}
\psi^{\pm}(r):=\int_{r}^{L}s^{d-1}\overbar{u}^{\pm}(s)\,ds=|\partial B_1|\cdot\|u^{\pm}\|_{L^1(r\leqslant|x|\leqslant L)}.
\end{equation*}
Dividing \eqref{eq:est-du/dr} by $r^{d-1}$ and integrating over $(r,L)$ there results that
\begin{equation}
\label{eq:est-u}
\overbar{u}(r)\leqslant \lambda \int_{r}^{L}\frac{\psi^{+}(s)}{s^{d-1}}\,ds-\lambda\int_{r}^{L}\frac{\psi^{-}(s)}{s^{d-1}}\,ds+C\big(\G_{\lambda}(r)+1\big).
\end{equation}
Multiplying \eqref{eq:est-u} by $r^{d-1}$, and observing that the second term in the right hand side of \eqref{eq:est-u} is negative, it follows that
\begin{equation*}
r^{d-1}\overbar{u}(r)\leqslant \lambda\int_{r}^{L}\psi^{+}(s)\,ds+Cr^{d-1}\big(\G_{\lambda}(r)+1\big).
\end{equation*}
Since $\psi^{+}$ is nonincreasing and $r^{d-1}(\G_{\lambda}(r)+1)$ is bounded on $(0,L)$, we get
\begin{equation}
\label{eq:est-rN-1u}
r^{d-1}\overbar{u}(r)\leqslant \lambda L \psi^{+}(r)+C,
\end{equation}
Integrating \eqref{eq:est-rN-1u} over the interval $(r,L)$, and using the fact that $\psi^{+}$ is nonincreasing  there results that 
\begin{equation}\label{eq:jk}
\psi^{+}(r)\leqslant\lambda L^{2}\psi^{+}(r)+\psi^{-}(r)+C.
\end{equation}
Since $u^{-}\in L_{\mathrm{loc}}^1(\R^d)$ by hypothesis (iii), then $\psi^{-}$ stays bounded as $r\to 0$. Hence by choosing $L$ sufficiently small we deduce from \eqref{eq:jk} that also $\psi^{+}$ stays bounded as $r\to 0$, which implies $u\in L^{1}_{\text{loc}}(\R^{d})$. Moreover, 
from the boundedness of $\psi^{\pm}$ for $r\in(0,R)$ and \eqref{eq:est-u} we also deduce that
\begin{equation}\label{medugr}
|\overbar{u}(r)|\lesssim \G_{\lambda}(r)+1\quad \forall\, r\in(0,R).
\end{equation}
\par\smallskip
\emph{Step 2: $\Delta u$ coincides a.e.~with a $L^1$-function.}
We define a measurable function $\varphi$ by the relation $\varphi:=-\Delta u$ a.e.~on  $\R^d$. By hypothesis (i), $\varphi\in L^1_{\mathrm{loc}}(\R^d\setminus\{0\})$. We are going to show that actually it belongs to $L^1(\R^d)$. To this aim, consider for $\eps>0$ functions $\zeta_{\eps}$ satisfying the following properties:
\begin{gather}
\zeta_{\varepsilon} \in \mathcal{C}^{\infty}(B_R), \quad 0 \leqslant \zeta_{\varepsilon} \leqslant 1 \mbox { on } B_R,\quad \zeta_{\varepsilon}(x)=0  \mbox { for }|x|<\varepsilon,\\	
\lim_{\eps\to 0}\zeta_{\varepsilon}(x)=1\quad\forall x\in B_R^*,\\\label{posila}
\Delta \zeta_{\varepsilon} \geqslant 0  \text { on } B_R.
\end{gather}
Observe that
\begin{equation}\label{phizeta}
	\begin{split}
\int_{B_R}\varphi\zeta_{\eps}dx&=-\int_{B_R}u\Delta\zeta_{\eps}\lesssim \int_{B_R}\mathcal{G}_{\lambda}\Delta\zeta_{\eps}dx\\
&=\lambda\int_{B_R}\mathcal{G}_{\lambda}\zeta_{\eps}dx\lesssim \int_{B_R}\mathcal{G}_{\lambda}<+\infty,
\end{split}
\end{equation}
where in the second step we used \eqref{posila} and hypothesis (iii). Moreover, by \emph{Step 1} and hypothesis (ii), we have $\varphi\geqslant -\lambda u+g\in L^1(B_R)$. Thus by \eqref{phizeta}  and Fatou Lemma we obtain $\varphi\in L^1_{\mathrm{loc}}(\R^d)$, as desired.\par\smallskip

\emph{Step 3: the distributional identity \eqref{eq:distr-form-delta} holds.} Let now $T$ be the distribution on $\R^d$ given by
\begin{equation*}
T=-\Delta u-\varphi,
\end{equation*}
whose support is contained in $\{0\}$. In particular, this entails that 
\begin{equation}
\label{eq:T=delta}
T=\sum_{|m|\leqslant k} c_{m}D^{m}\delta_{0},
\end{equation}
for some $k\in\N$, where $m=(m_{1},\dots,m_{d})$ is a multi-index, $|m|:=\sum_{i=1}^{d}m_{i}$, and \begin{equation*}D^{m}=\frac{\partial^{|m|}}{\partial x_{1}^{m_{1}}\dots\partial x_{d}^{m_{d}}}\end{equation*} in the sense of distributions. We claim that $c_{m}=0$ for every $m$ such that $|m|\geqslant 1$. Choose $\eta\in \mathcal{C}^{\infty}_{0}(B_{R})$ satisfying
\begin{equation*}
(-1)^{|m|}(D^{m}\eta)(0)=c_{m}\qquad \forall\quad |m|\leqslant k.
\end{equation*}
 Let $\eta_{\eps}(x):=\eta\left(\frac{x}{\eps}\right)$ for $\eps>0$. By \eqref{eq:T=delta}, and by the definition $T=-\Delta u-\varphi$, it follows that
\begin{equation}
\label{eq:-ulapeta}
\begin{split}
-\int_{B_{R}}u\Delta \eta_{\eps}\,dx-\int_{B_{R}}\varphi\eta_{\eps}\,dx
&=\langle T,\eta_{\eps}\rangle\\
&=\sum_{|m|\leq k}c_m(-1)^{|m|}(D^{m}\eta_{\eps})(0)\\
&=\sum_{|m|\leq k}\frac{c_{m}^{2}}{\eps^{|m|}}.
\end{split}
\end{equation}
Since $\varphi\in L^1(B_R)$, as proved in \emph{Step 2}, by dominated convergence we have 
\begin{equation}\label{ipin}
	\lim_{\eps\to 0}\int_{B_R}\varphi\eta_{\eps}\,dx=0.
\end{equation}
On the other hand, in view of the identity
\begin{equation*}
\int_{B_{R}}u\Delta \eta_{\eps}\,dx=\frac{1}{\eps^{2}}\int_{B_{\eps R}}u(\Delta \eta)\left(\frac{x}{\eps}\right)\,dx,
\end{equation*}
it holds that
\begin{equation}
\label{eq:mod-ulapeta}
\begin{split}
\left|\int_{B_{R}}u\Delta\eta_{\eps}\,dx\right| \leqslant & \,\frac{\|\Delta\eta\|_{L^{\infty}(B_{R})}}{\eps^{2}}\int_{B_{\eps R}}|u|\,dx\sim \frac{1}{\eps^{2}}\int_{0}^{\eps R}r^{d-1}|\overbar{u}(r)|\,dr\\
\lesssim & \frac{1}{\eps^{2}}\int_{0}^{\eps R}r^{d-1}(\G_{\lambda}(r)+1)\,dr\\
\lesssim &
\begin{cases}
1,\quad &d>2,\\
|\ln \eps|+1,\quad &d=2,
\end{cases}
\end{split}
\end{equation}
where we used estimate \eqref{medugr} in the third step. By combining \eqref{eq:-ulapeta}, \eqref{ipin} and \eqref{eq:mod-ulapeta}, it follows that
\begin{equation*}
\left|\sum_{|m|\leqslant k}\frac{c_{m}^{2}}{\eps^{|m|}}\right|\lesssim
\begin{cases}
1\quad &d>2,\\
|\ln\eps|+1,\quad &d=2,
\end{cases}
\end{equation*}
entailing that $c_{m}=0$ if $|m|\geqslant 1$ and concluding the proof.\medskip

\noindent\textbf{Acknowledgments} Filippo Boni and Raffaele Scandone acknowledge partial support from INdAM-GNAMPA. Diego Noja acknowledges the support of the Next Generation EU - Prin 2022 project ``Singular Interactions and Effective Models in Mathematical Physics-2022CHELC7'' and of Gruppo Nazionale di Fisica Matematica (GNFM-INdAM).

\medskip

\noindent\textbf{Data Availability} No datasets were generated or analyzed during the current study.

\section*{Declarations}

\noindent\textbf{Conflict of interest} On behalf of all authors, the corresponding author states that there is no conflict of interest.

\end{document}